\documentclass[12pt]{amsart} 
\usepackage{geometry}             

\usepackage{float}
\usepackage{bm}
\usepackage{fullpage,xcolor}
\usepackage[mathscr]{euscript}
\usepackage[all]{xy}
\usepackage{epsfig}
\usepackage[OT2,T1]{fontenc}
\newcommand\textcyr[1]{{\fontencoding{OT2}\fontfamily{wncyr}\selectfont #1}}

\usepackage{amsfonts}

\usepackage{graphicx}
\usepackage{amssymb}
\usepackage{amsmath}
\usepackage{amsthm}
\usepackage{mathrsfs}
\usepackage{epstopdf}
\usepackage{url}
\usepackage{hyperref}
\usepackage[msc-links,alphabetic]{amsrefs}
\usepackage{upgreek}

\usepackage[OT2,OT1]{fontenc} \newcommand\cyr
{
\renewcommand\rmdefault{wncyr} \renewcommand\sfdefault{wncyss} \renewcommand\encodingdefault{OT2} \normalfont
\selectfont
}
\DeclareTextFontCommand{\textcyr}{\cyr}

\usepackage{color}
\usepackage{subfig} 

\newcommand{\al}{\alpha}
\newcommand{\be}{\beta}
\newcommand{\ga}{\gamma}

\newcommand{\ep}{\varepsilon}

\renewcommand{\phi}{\varphi}

\newcommand{\om}{\omega}

\newcommand{\Om}{\Omega}
\newcommand{\De}{\Delta}
\newcommand{\Si}{\Sigma}
\newcommand{\ZZ}{{\mathbb Z}}

\newcommand{\RR}{{\mathbb R}}

\newcommand{\cE}{\mathscr E}

\newcommand{\Zh}{\hbox{\hspace{-7.0mm} \textcyr{Zh}}} 
\newcommand{\zh}{\hbox{\hspace{-6.7mm} \textcyr{zh}}} 

\newcommand{\lto}{\longrightarrow}

\newcommand{\Int}{\operatorname{Int}}

\newcommand{\Tube}{\operatorname{Tube}}
\newcommand{\sm}{\smallsetminus}
\newcommand{\co}{\colon}

\newcommand{\wh}[1]{{\widehat{#1}}}

\newcommand*\wbar[1]{
  \hbox{ \kern-0.2em%
    \vbox{%
      \hrule height 0.5pt  
      \kern0.25ex
      \hbox{%
        \kern-0.15em
        \ensuremath{#1}%
        \kern-0.05em
      }%
    }%
  \kern0.05em}%
}

\newtheorem{theorem}{Theorem}[section]

\newtheorem{proposition}[theorem]{Proposition}

\newtheorem{corollary}[theorem]{Corollary}
\theoremstyle{definition}     
\newtheorem{definition}[theorem]{Definition}

\theoremstyle{remark}
\newtheorem{remark}[theorem]{Remark}
\newtheorem{example}[theorem]{Example}

\title[Virtual concordance and the generalized Alexander polynomial]
{Virtual concordance and the \\ generalized Alexander polynomial}

\author[H. U. Boden]{Hans U. Boden}
\address{Mathematics \& Statistics, McMaster University, Hamilton, Ontario}
\email{boden@mcmaster.ca}
\urladdr{https://www.math.mcmaster.ca/~boden}
\thanks{The first author was partially funded by the Natural Sciences and Engineering Research Council of Canada, and the second author was partially funded by The Ohio State University.}

\author[M. Chrisman]{Micah Chrisman}
\address{Mathematics, The Ohio State University, Columbus, Ohio}
\email{chrisman.76@osu.edu}
\urladdr{https://micah46.wixsite.com/micahknots}


\begin{document}\

\begin{abstract}
We use the Bar-Natan \ $\Zh$-correspondence to identify the generalized Alexander polynomial of a virtual knot with the Alexander polynomial of a two component welded link. We show that the \ $\Zh$-map is functorial under concordance, and also that Satoh's Tube map (from welded links to ribbon knotted tori in $S^4$) is functorial under concordance. In addition, we extend classical results of Chen, Milnor, and Hillman on the lower central series of link groups to links in thickened surfaces. Our main result is that the generalized Alexander polynomial vanishes on any knot in a thickened surface which is virtually concordant to a homologically trivial knot. In particular, this shows that it vanishes on virtually slice knots. We apply it to complete the calculation of the slice genus for virtual knots with four crossings and to determine non-sliceness for a number of 5-crossing and 6-crossing virtual knots.
\end{abstract}

\maketitle

\section{Introduction} \label{S1}

A knot in $S^3$ is called \emph{slice} if it bounds a smoothly embedded disk in $B^4$. Many classical knot invariants take a special form on slice knots. For instance,  
the knot signature vanishes, the determinant is a square integer, and the Alexander polynomial factors as $f(t) f(t^{-1})$ for some integral polynomial. (The last one is the famous Fox-Milnor condition.)

In this paper, we study virtual knots and establish a similar condition on the generalized Alexander polynomial
for slice virtual knots. The generalized Alexander polynomial first appeared in the work of Jaeger, Kauffman, and Saleur. In  \cite{JKS-1994} they introduce a determinant formulation for the Alexander-Conway polynomial derived from the free fermion model in statistical mechanics. Their approach involves constructing link invariants via partition functions and solutions to the Yang-Baxter equation. Using a particular solution, they derive an invariant of knots in thickened surfaces which takes the form of a Laurent polynomial in two variables. In the planar case, the invariant can be further normalized to give the classical Alexander-Conway polynomial. 

In subsequent work, Sawollek \cite{saw} observed that the approach of \cite{JKS-1994} leads to a well-defined invariant of oriented virtual knots and links  now known as the \emph{generalized Alexander polynomial}. Virtual knots were introduced by Kauffman in the mid 1990s, and they can be viewed as knots in thickened surfaces up to isotopy, diffeomorphism, and \emph{stabilization}, described in more detail in \S \ref{S2}. The generalized Alexander polynomial plays a prominent role in virtual knot theory and can be interpreted in a variety of ways. It can be expressed as the zeroth elementary ideal of the extended Alexander group \cite{Silver-Williams-2003}, in terms of quandles \cite{Kauffman-Radford, Manturov2002} and virtual biquandles \cite{Crans-Henrich-Nelson}, and as the Alexander invariant of the virtual knot group \cite{Boden-Dies-Gaudreau-2015}. It has numerous applications: it can detect non-invertibility and non-classicality \cite{saw, Silver-Williams-2003}; it gives a lower bound on the virtual crossing number \cite{Boden-Dies-Gaudreau-2015}; and it dominates the writhe polynomial \cite{Mellor-2016}.

A knot $K \subset \Sigma \times I$ in a thickened surface is said to be \emph{virtually slice} if it bounds a smooth 2-disk in $W \times I$, where $W$ is a compact oriented $3$-manifold with $\partial W=\Sigma$. There is much computational evidence suggesting that the generalized Alexander polynomial is a slice obstruction for virtual knots (see \cite[Conj.~4.1]{Boden-Chrisman-Gaudreau-2017}). Our main theorem confirms this conjecture by showing  that the  generalized Alexander polynomial vanishes on knots in thickened surfaces which are virtually concordant to homologically trivial knots (see Theorem \ref{thm_main}).  In particular, it follows that the generalized Alexander polynomial vanishes on virtually slice knots.

In \S \ref{S2}, we introduce some basic notions and sketch the argument of the main result.
In \S \ref{S3}, we introduce virtual boundary links, and we give virtual analogues of results of Milnor \cite[Theorem 4]{milnor} and Hillman \cite{hill}. 
In \S \ref{S4}, we present the $\Zh$-construction (due to Bar-Natan) and describe the Tube construction (due to Satoh). Both are shown to be functorial under concordance. The main result and its proof are given in \S \ref{subsec:main}, and applications to obstructing sliceness are given in \S \ref{subsec:applications}. The paper concludes with \S \ref{subsec:conclusion}, a brief discussion on further questions about the generalized Alexander polynomial and the $\Zh$-construction.

\section{Preparations}  \label{S2}  
\subsection{Basic Notions} 
We briefly review the basic notions of virtual knot theory, including virtual knot diagrams, welded equivalence, the link group, Gauss diagrams, and knots in thickened surfaces. For a thorough introduction, we recommend \cite{Kauffman-1999}.
Note that all virtual knots and links in this paper will be oriented.

A virtual knot diagram is an immersion of a circle in the plane whose double points are either classical (indicated by over- and undercrossings) or virtual (indicated by a circle). Figure \ref{fig:4-12} shows the virtual knot $4.12$. The decimal number refers to the virtual knot table  \cite{Green}. Virtual link diagrams are defined similarly. Two such diagrams are virtually equivalent if they are related by planar isotopies, (classical) Reidemeister moves, and the detour move (see Figure \ref{detour}). 
These moves are known collectively as the generalized Reidemeister moves.

\begin{figure}[h]
\centering
 \includegraphics[scale=1.35]{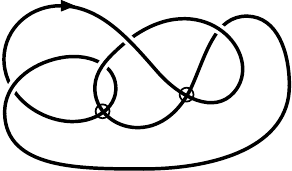}  \qquad \qquad \includegraphics[scale=0.9]{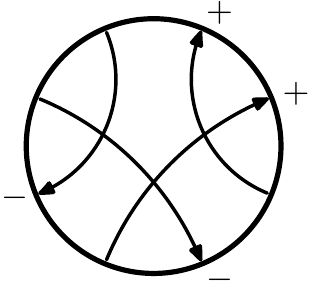} 
 \caption{{The virtual knot $4.12$ and its Gauss diagram. The Gauss code is {{\tt O1-O2-U1-O3+U2-O4+U3+U4+}} and can be read from either diagram. On left, start at the arrowhead; on right, go counterclockwise from 12 o'clock.}}
\label{fig:4-12}
\end{figure} 

Two virtual links are said to be \emph{welded equivalent} if they can be represented by diagrams related by a sequence of generalized Reidemeister moves and the forbidden overpass  (see Figure \ref{detour}). 

\begin{figure}[ht]
\centering
 \includegraphics[scale=0.90]{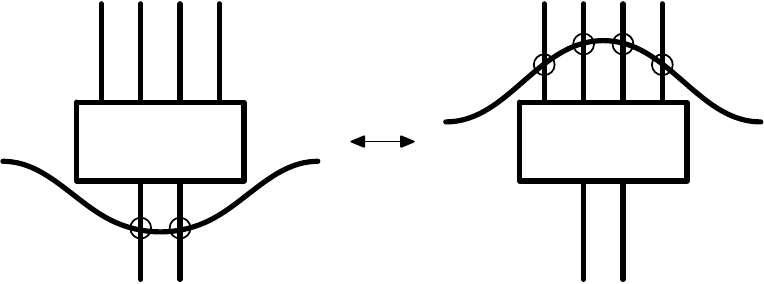} 
  \qquad \includegraphics[scale=0.90]{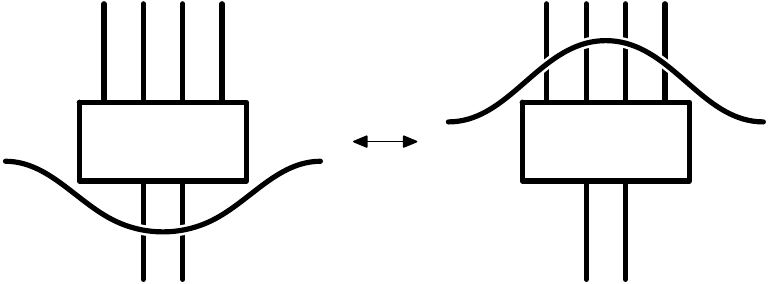} 
 \caption{The detour move (left) and forbidden overpass (right).}
\label{detour}
\end{figure}

Associated to any virtual knot diagram is a Gauss code, which is a word that records the over and undercrossing information and signs of the crossings as one goes around the knot. A Gauss diagram is a decorated trivalent graph that records the same information. Full details can be found in \cite{Kauffman-1999}, and Figure  \ref{fig:4-12} gives an illustrative  example.
Similar considerations apply to virtual link diagrams. There is a notion of equivalence of Gauss diagrams generated by generalized Reidemeister moves, and one can define virtual knots and links as equivalence classes of Gauss diagrams.  

One can alternatively define virtual links as links in thickened surfaces up to stable equivalence.
Let $\Si$ be a compact, connected, oriented surface. A link $L$ in  $\Si \times I$ is a 1-dimensional submanifold with finitely many components, each of which is homeomorphic to $S^1.$ Stabilization is the addition or removal of a 1-handle to $\Si$ disjoint from the diagram of $L$.
Specifically, given disks $D_0, D_1 \subset \Si$ disjoint from one another and from the projection of $L$ under
$\Si\times I\to \Si$, we set $\Si' =\Si \sm (D_0 \cup D_1) \cup S^1 \times I$, where the 1-handle $S^1 \times I$ is attached along  $\partial D_0 \cup \partial D_1$. This operation  is called \emph{stabilization},
and the opposite procedure, which is the removal of a 1-handle disjoint from $L$, is called \emph{destabilization}.
Two links $L_0 \subset \Si_0\times I$ and $L_1 \subset \Si_1 \times I$ are said to be \emph{stably equivalent} if one is obtained from the other by a finite sequence of isotopies, oriented diffeomorphisms of $\Si \times I$, stablizations, and destablizations. By \cite{Carter-Kamada-Saito}, a virtual link is uniquely determined by the associated stable equivalence class of links in thickened surfaces, and vice versa.

A link $L \subset \Si \times I$ is said to have \emph{minimal genus} if it cannot be destabilized. On \cite{kuperberg}, Kuperberg proved that any two minimal genus representatives for the same virtual link are equivalent up to isotopy and diffeomorphism.


\subsection{Generalized Alexander polynomial}
In \cite{saw},  Sawollek introduced the generalized Alexander polynomial, an invariant of oriented virtual links. His definition is an 
adaptation of work of Jaeger, Kauffman and Saleur \cite{JKS-1994}, who 
extended the Alexander-Conway polynomial to links in thickened surfaces.
Using the extended Alexander group, Silver and Williams defined Alexander invariants
for virtual links \cite{Silver-Williams-2001}, and 
in \cite{Silver-Williams-2003}, they showed how to interpret Sawollek's polynomial as the zeroth
order Alexander invariant of their extended Alexander group. The resulting invariant will be denoted 
$\De^0_L(s,t)$.  We recall its definition from \cite{Mellor-2016, Silver-Williams-2003}. (Note that we use variables $s,t$ in place of $v,u$, respectively.) 

Suppose $K$ is a virtual knot, represented as a virtual knot diagram with $n$ classical crossings. Fix an ordering of the crossings, and use it to determine an ordering on the short arcs $a_1,\ldots, a_{2n}$ of the diagram as follows. The short arcs start at one classical crossing and end at the next classical crossing, passing through any virtual crossings along the way. The ordering on the short arcs is chosen so that, at the $i$-th crossing, the incoming short arcs  are $a_{2i-1}$ and   $a_{2i}$, as depicted in Figure \ref{crossing-gap}. 

\begin{figure}[h]
\centering
  \includegraphics[scale=0.90]{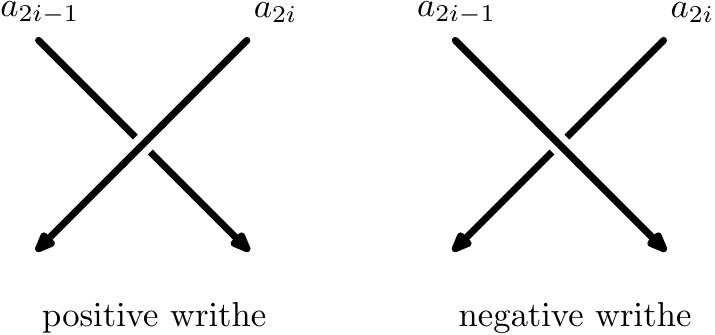}  
 \caption{The labeling of incoming short arcs at the $i$-th crossing.} \label{crossing-gap}
\end{figure} 

Define $M$ to be the block diagonal $2n \times 2n$ matrix with the $2 \times 2$ blocks $M_1,\ldots, M_n$ placed along the diagonal. Here $M_i$ is the matrix corresponding to the $i$-th crossing, and it is equal to $$M_+ = \begin{bmatrix} t^{-1} & 1-(st)^{-1} \\ 0 & s^{-1} \end{bmatrix} \quad \text{ or }
\quad M_- = \begin{bmatrix} s & 0 \\  1-st & t \end{bmatrix}$$
depending on the writhe of the $i$-th crossing; namely 
 we use $M_+$ if the writhe is positive and $M_-$ if it is negative.

The sequence of short arcs that are encountered as one travels along the knot determines a permutation of $\{a_1,\ldots, a_{2n}\}$, given as a cycle.
Let $P$ denote the associated permutation matrix; its $ij$ entry is equal to 0 unless $a_i$ precedes $a_j$ as consecutive short arcs, in which case the $ij$ entry equals 1. The generalized Alexander polynomial is then given by $\De^0_K(s,t) = \det(M-P)$, and it is an invariant of the virtual link up to multiplication by powers of $st$. 

\begin{example} \label{ex:4-12}
For the virtual knot $4.12$ in Figure \ref{fig:4-12}, one can compute that its generalized Alexander polynomial is given by
$\De^0_K(s,t)= (1-t)(1-s)(t - s)(1- st)^2$.
\end{example}

 In \cite{Silver-Williams-2001}, Silver and Williams defined a multi-variable generalized Alexander polynomial. For a virtual link $L$ with $\mu$ components, 
we will denote their invariant by $\De_L^0(s,t_1,\ldots, t_\mu)$.
Note that we use variables $s,t_1,\ldots, t_\mu$ in place of $v,u_1,\ldots, u_\mu$ (cf. \cite[\S 4]{Silver-Williams-2001}).

Using a simple change of variables, Silver and Williams showed that 
$\De_L^0(s,t_1,\ldots, t_\mu)$ vanishes whenever $L$ is classical (see \cite[Corollary 4.3]{Silver-Williams-2001}). In \cite[\S 4]{Silver-Williams-2006}, they extended this vanishing result to include \emph{almost classical} links. Recall that
a virtual link is said to be almost classical if some diagram of it admits an Alexander numbering
\cite[Definition 4.1]{Silver-Williams-2006}. 
Note further that a virtual link is almost classical if and only if it can be represented as a homologically trivial link in a thickened surface (see \cite[Theorem 6.1]{acpaper}).

The vanishing result for $\De_L^0(s,t_1,\ldots, t_\mu)$ can also be deduced Theorem 5.2 \cite{acpaper}. 
In \S \ref{S4-1}, we will give a new proof that the generalized Alexander polynomial vanishes for almost classical links, and in \S \ref{subsec:main}, we strengthen the statement and show that the generalized Alexander polynomial vanishes for any virtual knot that is concordant to an almost classical knot.


\subsection{Virtual link groups} \label{subsec:vlg}
Suppose $D$ is a virtual link diagram  with $\mu$ components $D_1,\ldots, D_\mu$.
Suppose further that the $i$-th component $D_i$ has $n_i$ undercrossings for $i=1,\ldots, \mu,$
thus $D$ has a total of $n=n_1 + \cdots + n_\mu$ crossings. 
We will label the arcs of $D$ according to the following scheme.
 
Fix a base point on $D_i$ that is not a crossing point, and
starting at the base point, label the arcs $a_{i1},\ldots a_{in_i}$ consecutively so that $a_{i1}$ contains the basepoint and so that $a_{ij}$ and $a_{i j+1}$ are the incoming and outgoing undercrossing arcs, respectively, at the $j$-th undercrossing (see Figure \ref{crossing-reg}). At the last crossing of $D_i$, $a_{in_i}$ will be the incoming undercrossing arc and $a_{i1}$ will be the outgoing undercrossing arc. Thus, we take $j$ modulo $n_i$.
Let $a_{k \ell}$ be  the overcrossing arc at this crossing, and let $\ep_{ij}$ be the sign of the crossing. Notice that $k=k(ij)$ and $\ell = \ell(ij)$ are both functions of $i$ and $j$, but we suppress their dependence here.

 \begin{figure}[h]
\centering
 \includegraphics[scale=0.90]{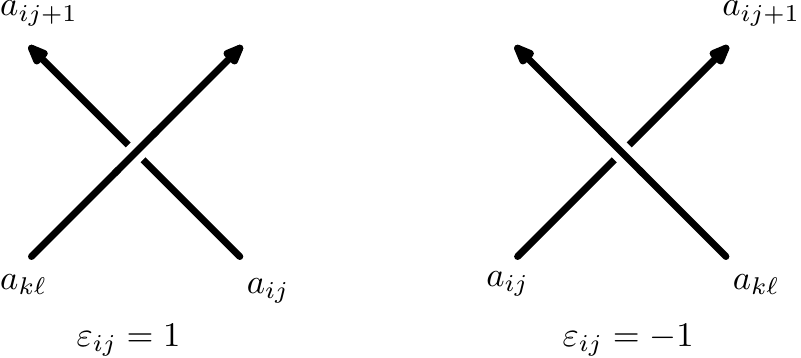} 
 \caption{The arc labels at a positive and negative crossing.} \label{crossing-reg}
\end{figure} 

\begin{definition} \label{defn:link-group}
The \emph{link group} is the  group   with generators $a_{ij}$ for $1\leq i \leq \mu$ and $1 \leq j \leq n_i$, and with relations given by the presentation
\begin{equation} \label{eqn:link-group}
G_D=\left<a_{ij} \mid a_{ij+1}=(a_{k\ell})^{-\ep_{ij}} \, a_{ij} \, (a_{k\ell})^{\ep_{ij}}, i=1,\ldots,\mu, j=1, \ldots, n_i \right>.
\end{equation}
\end{definition}

This group is invariant under the generalized Reidemeister moves and the forbidden overpass, and consequently it is an invariant of the welded type of $D$ (see \cite{Kim-2000}).  

For virtual links represented instead as links $L \subset \Si \times I$ in thickened surfaces, the link group $G_L$ is the fundamental group of the space $Y$ obtained from $\Si \times I\sm L$ by collapsing $\Si \times \{1\}$ to a point (see \cite[Proposition 5.1]{Kamada-Kamada-2000}). 

Given a link group $G=G_L$ for an oriented link $L$ with $\mu$ components,
a choice of meridians $m_1, \ldots, m_\mu$ determines an epimorphism
$\al \co G \to \ZZ^\mu$ with kernel $G'=[G,G]$, the commutator subgroup. Setting
$G''=[G',G']$,  we regard the quotient $G'/G''$ as a module over the group ring $\ZZ[\ZZ^\mu] = \ZZ[t_1^{\pm 1},\ldots, t_\mu^{\pm 1}]$.

This module is of course just the Alexander invariant of $L$, and using Fox differentiation one can easily describe it in terms of a presentation matrix $A$ called the \emph{Alexander matrix}.
For instance, for the group presentation \eqref{eqn:link-group}, the Alexander matrix is the $n \times n$ matrix  with $ij$ entry given by $\left(\frac{\partial r_{ij}}{\partial a_j}\right)^\al,$
where $r_{ij}$ is the relation $(a_{k\ell})^{\ep_{ij}} a_{ij} (a_{k\ell})^{-\ep_{ij}} (a_{i j+1})^{-1}.$
 The $k$-th elementary ideal is denoted $\cE_k$ and defined to be the ideal generated by all the $(n-k) \times (n-k)$ minors of $A$ for $0\leq k \leq n-1$.
For $k=n,$ we set $\cE_n =\ZZ[t_1^{\pm 1},\ldots, t_\mu^{\pm 1}].$

The virtual link group $VG_L$ was introduced in \cite{Boden-Dies-Gaudreau-2015}, and the
virtual Alexander polynomial, denoted ${H}_L (s, t, q)$, is an invariant derived from the zeroth elementary ideal of the virtual Alexander invariant of $L$ (see \cite[Definition 3.1]{Boden-Dies-Gaudreau-2015}). It is closely related to the generalized Alexander invariant. Indeed,  Proposition 3.8 of \cite{Boden-Dies-Gaudreau-2015} implies that $\De^0_L(s,t)= {H}_L (s, t, 1)$ up to powers of $s$ and $t$. The virtual Alexander polynomial also admits a normalization, which is denoted $\wh{H}_L(s,t,q)$ and well-defined up to powers of $st$, and Proposition 5.5 of \cite{Boden-Dies-Gaudreau-2015} implies that the normalized version satisfies $\wh{H}_L (s, t, q) = \wh{H}_L (sq^{-1}, tq, 1)$. Thus the virtual Alexander polynomial determines the generalized Alexander polynomial and vice versa.

Next, we introduce the \emph{reduced virtual link group} $\wbar{G}_L$. In \S \ref{S4}, we will use the Bar-Natan $\Zh$-construction to relate $\wbar{G}_L$ to the link group of the welded link $\Zh(L)$.

\begin{definition} \label{defn-red-virt-gp}
Given a virtual link $L$ represented as a virtual link diagram $D$, one can associate a finitely generated group with one generator for every short arc of $D$, together with one auxiliary generator $v$ and one relation for each classical crossing of $D$ as given in Figure \ref{crossing-red}.

\begin{figure}[h]
\centering
 \includegraphics[scale=0.90]{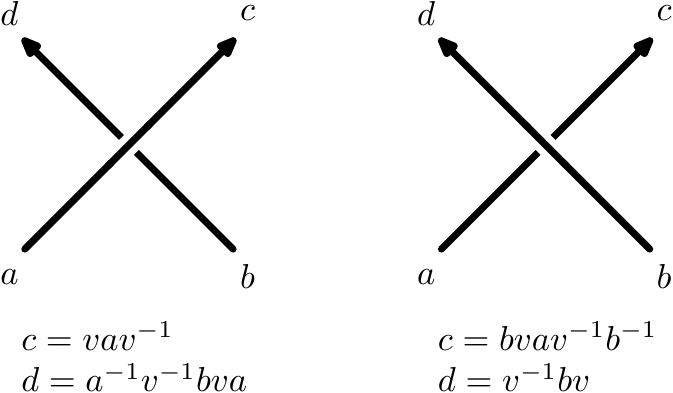} 
 \caption{Relations for the reduced virtual link group $\wbar{G}_L$.} \label{crossing-red}
\end{figure} 

The resulting finitely presented group is an invariant of the virtual link and is denoted
$\wbar{G}_L$ (see \cite[Section 3]{acpaper} for a proof of invariance). It is called the \emph{reduced virtual link group}.
\end{definition} 

We have chosen to use right-handed meridians here, and for that reason the isomorphism between the presentations of $\wbar{G}_L$ and the one given in \cite{acpaper} involves taking the inverses of the generators $a,b,c,d.$ The auxiliary generator $v$ is unchanged.

The group obtained from $\wbar{G}_L$ by taking $v=1$ is isomorphic to the link group $G_L$
\cite[\S 2]{acpaper}.

\begin{example} \label{ex:4-12}
For the virtual knot $4.12$ in Figure \ref{fig:4-12}, one can compute that its knot group $G_K$ is cyclic, and that its reduced virtual knot group $\wbar{G}_K$ is generated by two elements $a,v$ with a single relation: 
$$a=v^{-1} a v^3 a v^{-1} a^{-1} v a^{-1} v^{-4} a^{-1} v a^{-1} v^{-1} a v^3 a v^{-1} a v a^{-1} v^{-3} a^{-1} v a v^{-1} a v^4 a v^{-1} a v a^{-1} v^{-3} a^{-1} v.$$
\end{example}

By \cite[Theorem 3.1]{acpaper}, it follows that $VG_L \cong \wbar{G}_L *_\ZZ \ZZ^2$, thus the virtual link group can be recovered from $\wbar{G}_L$. Consequently, the Alexander invariants of these two groups are equivalent. The reduced Alexander polynomial $\wbar{H}_K (s, t)$ is defined in \cite[\S4]{acpaper} as the zeroth Alexander polynomial associated to the reduced virtual Alexander invariant $\wbar{G}_L'/\wbar{G}_L''$. It is  defined to be the generator of the smallest principal ideal containing $\wbar{\cE}_0$, the zeroth elementary ideal. Then \cite[Theorem 4.1]{acpaper} implies that $\wbar{H}_K (s, t)$ determines the generalized Alexander polynomial and vice versa.


\subsection{Concordance}
In this section, we introduce two equivalent notions of concordance for links in thickened surfaces and for virtual links, one is due to Turaev \cite{Turaev-2008} and the other to Kauffman \cite{Kauffman-2015}. We present several examples of virtual knots that are slice, and one that shows that the generalized Alexander polynomial is not, in general, invariant under virtual concordance. At the end, we define welded link concordance.

In general, two links in a 3-manifold $M$ are said to be \emph{cobordant} if they cobound a properly embedded oriented surface in the cylinder $M \times I.$ They are said to be \emph{concordant} if, in addition, the cobordism surface can be chosen to be a disjoint union of annuli. 

In the case $M$ is a thickened surface $\Si \times I$, there are more relaxed notions that are called virtual cobordism and virtual concordance. They were first introduced by Turaev \cite{Turaev-2008}, and they provide the added flexibility of allowing the surface $\Si$ to change by cobordism.

\begin{definition}\label{defn:virt-link-conc}
Two oriented links  $L_0  \subset \Si_0 \times I$ and $L_1  \subset \Si_1 \times I$ in thickened surfaces are said to be \emph{virtually cobordant} if there exist a compact oriented 3-manifold $W$ with $\partial W = -\Si_0 \cup \Si_1$, together with an oriented surface $Z$ smoothly and properly embedded in $W \times I$ with $\partial Z = -L_0 \cup L_1$. 

The links $L_0$ and $L_1$ are said to be \emph{virtually concordant} if, in addition, the surface $Z \subset W \times I$ is a disjoint annuli, each with one boundary component on $\Si_0 \times I$ and the other on $\Si_1 \times I$. 
\end{definition}

The above definitions can be easily translated into the language of virtual links. For instance, two virtual links are said to be \emph{cobordant} if they can be represented by links in thickened surfaces that are virtually cobordant, and likewise for concordance. 

In \cite{Kauffman-2015}, Kauffman introduced an equivalent description for cobordism and concordance of virtual links. His definition is purely diagrammatic and is given in terms of finding a sequence of births, deaths, and saddle moves.

A birth is the addition of a trivial unknotted unlinked component and a death is the removal of a trivial unknotted unlinked component. A saddle is a move that cuts and rejoins two arcs in a way that respects their orientations, see Figure \ref{saddle}. 
Given a cobordism, its Euler number is defined to be $b+d-s$, where $b$ is the number of births, $d$ is the number of deaths, and $s$ is the number of saddles. The genus of the cobordism surface is determined from its Euler characteristic $\chi = b+d-s$ in the usual way.

\begin{figure}[h]
\centering
  \includegraphics[scale=1.20]{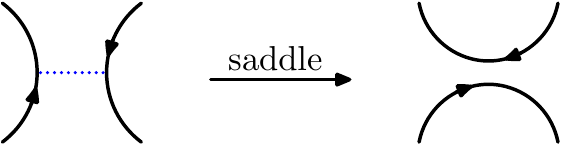}  
 \caption{A saddle move.} \label{saddle}
\end{figure} 

\begin{definition} \label{defn:cobordism}
Two virtual links $L$ and $L'$ are said to be \emph{cobordant} if they have the same number of components and can be represented by virtual link diagrams that are related to one another by a finite sequence of moves that include generalized Reidemeister moves, together with births, deaths, and saddles. 

Further, two virtual links $L=K_1 \cup \cdots \cup K_\mu$ and $L'=K'_1 \cup \cdots \cup K'_\mu$ are said to be \emph{concordant} if there exists a sequence of generalized Reidemeister moves and $b_i$ births, $d_i$ deaths, and $s_i$ saddles which take $K_i$ to $K_i'$ and satisfy $b_i +d_i =s_i$ for each $1 \leq i \leq \mu$. A virtual link that is concordant to the unlink is called \emph{slice}. 
\end{definition} 

\begin{remark} 
(i) Note that cobordant virtual links are required to have the same number of components. Likewise for concordance. \\
(ii)  Cobordism determines an equivalence relation on virtual links. The pairwise virtual linking numbers of the components are invariant under virtual link cobordism. Concordance also determines an equivalence relation on virtual links which is of course a refinement of cobordism. \\
(iii)  Any two virtual knots are cobordant, in particular, any virtual knot is cobordant to the unknot (see \cite{Kauffman-2015}).  The \emph{slice genus} of a virtual knot $K$ is denoted $g_s(K)$ and is defined to be the minimum genus over all cobordisms from $K$ to the unknot.  
\end{remark}

A proof that Definitions \ref{defn:virt-link-conc} and \ref{defn:cobordism} are equivalent for virtual links was given by Carter, Kamada, and Saito \cite{Carter-Kamada-Saito-2004}. In these terms, they prove that two links in thickened surfaces are virtually concordant if and only if they represent concordant virtual links. 

As for classical knots, sliceness and concordance is most easily described via a slice movie where one shows the saddles, births and deaths diagrammatically. Suppose $L_0 \subset \Si_0 \times I$ and $L_1 \subset \Si_1 \times I$ are virtually concordant. Thus they represent concordant virtual links. By Definition \ref{defn:virt-link-conc}, there exists an oriented 3-manifold $W$ with $\partial W = -\Si_0 \cup \Si_1$ and an oriented surface $Z \subset W \times I$  with $\partial Z = -L_0 \cup L_1$. One can view virtual concordance in terms of relative Morse theory as follows. A concordance movie gives rise to a Morse function on $f\colon W \times I \to [0,1]$ with $f^{-1}(i) = \Si_i \times I$ for $i=0,1$ such that the restriction $f|_Z \colon Z \to [0,1]$ is also Morse. Critical points of $f$ correspond to stabilizations and destabilizations of the ambient surface. Local minima and maxima of $f|_Z$ correspond to births and deaths, respectively, and saddles of $f|_Z$ correspond to saddle moves. One can arrange the critical values of $f$ and $f|_Z$ to be disjoint by performing the generalized Reidemeister moves in between all births, deaths, and saddles. We give a few examples to illustrate this technique.

\begin{example} \label{ex:satoh}
Satoh's knot is depicted on the left of Figure \ref{fig:satoh}.  Performing a saddle along the dotted blue arc, one obtains the two-component virtual link on the right. Using generalized Reidemeister moves (actually type II only, both classical and virtual), this link is easily seen to be equivalent to the unlink of two components. Capping off one of the trivial components, one obtains a concordance from Satoh's knot to the unknot. Thus, Satoh's knot is seen to be slice.
\end{example}

\begin{figure}[h]
\centering
  \includegraphics[scale=0.85]{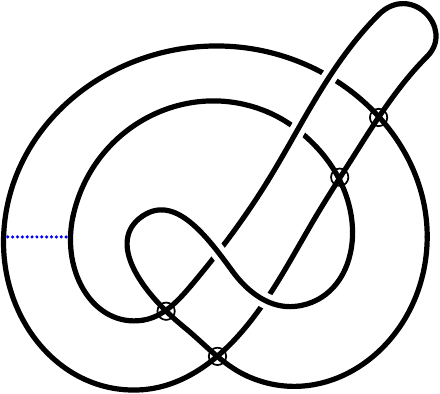} \qquad \qquad
  \includegraphics[scale=0.85]{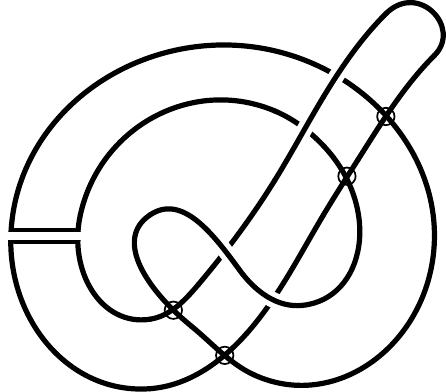}  
 \caption{Satoh's knot is slice.} \label{fig:satoh}
\end{figure} 

The same method can be used to show sliceness for many other virtual knots. For example, consider the virtual knots 4.59 and 4.98 in Figure \ref{fig:4-59}. Performing saddle moves along the dotted blue arcs and applying generalized Reidemeister moves, one can see that the resulting two component virtual links are trivial. Thus both virtual knots 4.59 and 4.98 in Figure \ref{fig:4-59} are slice. 

This technique for slicing virtual knots has been adapted to Gauss diagrams, and it allows one to quickly and easily find slicings for a great many virtual knots \cite{Boden-Chrisman-Gaudreau-2017}.

\begin{figure}[h]
\centering
  \includegraphics[scale=2.00]{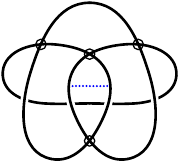} \qquad \qquad \includegraphics[scale=0.90]{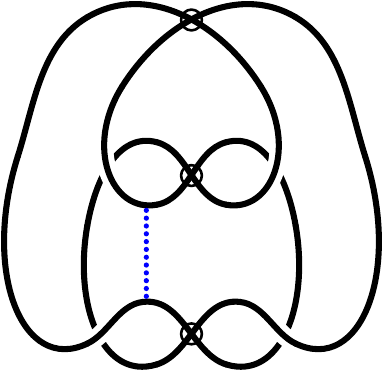}  
 \caption{The virtual knots $4.59$ (left) and $4.98$ (right) are both slice.} \label{fig:4-59}
\end{figure} 

There is a close relationship between concordance of virtual knots and connected sum. Recall that connected sum is not a well-defined operation on virtual knots; it depends on the diagrams used as well as the connection site. In fact, one can produce examples of nontrivial virtual knots by taking the connected sum of two virtual knot diagrams of the unknot. The Kishino knot is the most famous such example, and any virtual knot like that is slice; just perform a saddle move across the isthmus of the connected sum and the resulting two component virtual link is obviously trivial. (Note that this argument works more generally for connected sums of slice virtual knots.)

In a similar way, one can show that the virtual knots $K$ and $K'=K\# J$ are concordant whenever $J$ is slice. (This statement is independent of the indeterminacy inherent to connected sum operation on virtual knots.) 

\begin{figure}[h]
\centering
 \includegraphics[scale=2.00]{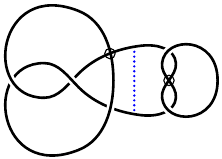} \qquad \qquad
 \includegraphics[scale=2.00]{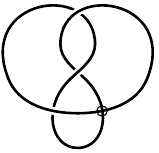} 
 \caption{The virtual knot $5.343$ (left) and the virtual knot 3.2 (right) are concordant.}\label{GD-3-2}
\end{figure} 

For a specific example, consider the virtual knot $5.343$ pictured on the left in Figure \ref{GD-3-2}. Performing a saddle along the dotted blue curve gives a concordance to the virtual knot 3.2, the virtual figure eight knot. (This is immediate from the fact that  $5.343$ is the connected sum of $3.2$ with a virtual unknot.) It follows that the virtual knots $5.343$ and $3.2$ are concordant. Furthermore, by direct computation, one sees that their generalized Alexander polynomials are given by 
\begin{eqnarray*}
\De_{3.2}^0(s,t) &=& (t^2 - 1) (s - 1)(st - 1), \text{ and } \\
\De_{5.343}^0(s,t) &=& - (t - 1)(s - 1)(st - 1)(s^2t + st - s).
\end{eqnarray*}
Since $\De_{3.2}^0 \neq \De_{5.343}^0$ up to units in $\ZZ[s^{\pm 1}, t^{\pm 1}],$ 
this example shows that the generalized Alexander polynomial is not invariant under concordance.

Going back to Definition \ref{defn:virt-link-conc}, when the cobordism is a product, i.e. when $W = (\Si \times I) \times I$, we say that the concordance is \emph{strict}. More precisely, two links $L_0, L_1  \subset \Si \times I$ are said to be \emph{strictly concordant} if there exists a disjoint union of annuli $Z \subset (\Si \times I) \times I$ with $\partial Z = -L_0 \cup L_1$. A link in $\Si \times I$ that is strictly concordant to the trivial link is said to be \emph{strictly slice}.

For links in thickened surfaces, strict concordance can be described in terms of saddles, births and deaths on a link diagram in the surface. However, because stabilization and de-stabilization of the surface is not permitted, strict concordance is considerably stronger than virtual concordance.
For example, if $L_0, L_1  \subset \Si \times I$ are strictly concordant, then $L_0$ and $L_1$ are homologous in $\Si \times I$. In particular, if one of them is null-homologous, then the other one must also be. For example, if $K \subset \Si \times I$ is a strictly slice knot, then $K$ is necessarily null-homotopic and therefore null-homologous.

On the other hand, there are many examples of knots in $\Si \times I$ which are virtually slice but not homologically trivial. The simplest ones are obtained by taking $K \subset \Si$ a simple closed curve with homology class $[K] \in H_1(\Si)$. As long as $[K]$ is non-trivial, then $K$ is not strictly slice, but it nevertheless represents the trivial virtual knot and hence is virtually slice.

More interesting examples can be obtained by taking a virtual knot that is slice but not almost classical. For instance, any of the knots from Figures \ref{fig:satoh} and \ref{fig:4-59}. When represented as a knot in a thickened surface, one obtains a knot in a thickened surface of genus two which is virtually slice but not strictly slice. The same is true of any of the many virtual knots from \cite{Boden-Chrisman-Gaudreau-2017} that are known to be slice but not almost classical. 
 
We end this section with the definition of welded concordance.
 
\begin{definition} \label{defn:welded-concordance}
Two virtual links $L=K_1 \cup \cdots \cup K_\mu$ and $L'=K'_1 \cup \cdots \cup K'_\mu$ are said to be \emph{welded concordant} if there is a finite sequence of generalized Reidemeister moves, forbidden overpass moves, and $b_i$ births, $d_i$ deaths, and $s_i$ saddles, which satisfy $b_i +d_i =s_i$ and take take $K_i$ to $K_i'$ for each $1 \leq i \leq \mu$.
\end{definition}

Welded concordance gives an equivalence relation on virtual links which is coarser than virtual concordance. For example, any virtual knot is welded concordant to the unknot. In short, every welded knot is slice. This is not obvious, but it follows by a clever argument due to Gaudreau \cite[Theorem 5]{Gaudreau-2019}.


\subsection{Sketch of the proof} \label{subsec:sketch}
Our main result is Theorem \ref{thm_main}. It states that the generalized Alexander polynomial vanishes for knots in thickened surfaces which are virtually concordant to homologically trivial knots. Below we give a brief sketch of the proof. 
 
Given a virtual knot $K$, we associate to it a two-component semi-welded link $\Zh(K),$
such that the (reduced) virtual knot group $\wbar{G}_K$ is isomorphic to the classical link group
$G_{\zh(K)}.$ It follows that their Alexander modules are isomorphic. Thus, they have  the same elementary ideals up to an overall degree shift, which is the result
of the fact that the number of meridional generators increases by one under the $\Zh$-map. (Note that the elementary ideals $\bar{\cE}_k$ of the virtual Alexander module
$\wbar{G}_K'/\wbar{G}_K''$ are indexed as in \cite{acpaper}.) 
In particular, the isomorphism of Alexander modules identifies $\bar{\cE}_k$ with $\cE_{k+1}$ for $k \ge 0.$

Additionally, the correspondence $K \mapsto \Zh(K)$ is functorial with respect to concordance.
Namely, if $K$ and $K'$ are concordant as virtual knots, then $\Zh(K)$ and $\Zh(K')$
are concordant as welded links.

Next, we apply Satoh's Tube map to the welded link $L= \Zh(K)$ (see \S \ref{S4} for more details), and consider the resulting
ribbon torus link $\Tube(L) \subset S^4.$  
The Tube map is also shown to be functorial in concordance. 
That is, if $L$ and $L'$ are concordant as welded links, then 
$\Tube(L)$ and $\Tube(L')$ are concordant as ribbon torus links. 

A result of Stallings then shows that, for compact submanifolds $T$ of $S^4$, the nilpotent quotients of the fundamental group of
$S^4 \sm T$ are invariant under concordance.
Applied to $T = \Tube(L),$ since
the classical link group $G_L$ is isomorphic to $\pi_1(S^4 \sm T),$ 
this implies that the nilpotent quotients of $G_L$ are 
invariant under concordance of the welded link $L.$

Any knot in a thickened surface represents a virtual knot $K$. If $K$ is concordant to an almost classical knot, then the welded link $L = \Zh(K)$ is welded concordant to a virtual boundary link (see Definition \ref{defn:VBL} below). 
Therefore, the nilpotent quotients of $G_L \cong \wbar{G}_K$ coincide with those of the free group $F(2)$ on two generators. Using the Chen-Milnor presentation of the nilpotent quotients, we conclude that the longitudes of $L$ must lie in the nilpotent residual $(G_L)_\om.$
Applying a result of Hillman, suitably extended to virtual link groups, it follows that the first elementary ideal $\cE_1$ is trivial, and we conclude that the generalized Alexander polynomial of $K$  must vanish.


\section{Classical results for virtual links} \label{S3} 
\subsection{Virtual boundary links} 
In the following, $\Si$ is a closed oriented surface, and $F(\mu)$ denotes the free group on $\mu$ generators.

\begin{definition} \label{defn:VBL}
A link $L  \subset \Si \times I$  
is
called a \emph{boundary link} if its components bound pairwise disjoint Seifert surfaces in $\Si \times I$.
A virtual link that can be represented by a boundary link  $L \subset \Si \times I$  will be called a \emph{virtual boundary link}.  
\end{definition}

Note that every virtual boundary link is almost classical, but the converse is not true for links with $\mu>1$ components. For links in thickened surfaces, the property of being a boundary link is not preserved under stabilization, but it is preserved under destabilization. This last fact can be proved using an argument similar to the proof of Theorem 6.4 in \cite{acpaper}. Thus, given a virtual link, it is a virtual boundary link if and only if its minimal genus representative satisfies Definition \ref{defn:VBL}. In particular, it follows that a classical link is a virtual boundary link if and only if it is a boundary link.

For a classical link $L \subset S^3$ with $\mu$ components, Smythe showed that $L$ is a boundary link if and only if there is an epimorphism $G_L \to F(\mu)$ mapping the set of meridians to a free basis \cite{Smythe-1966}. This result was extended to higher dimensional links by Guti\'{e}rrez \cite{Gutierrez-1972}. The next result gives the analogue in one direction for virtual links. Note that the other implication is not true for virtual links, in fact it fails already for virtual knots. 
In that case, the knot group $G_K$ always admits an epimorphism to $\ZZ,$ but a virtual knot admits a Seifert surface if and only if it is almost classical. Thus, a virtual knot is a boundary link if and only if it is almost classical.

\begin{proposition}\label{prop_epi} Suppose $L\subset \Si \times I$ is a link in a thickened surface with link group $G_L$. If $L$ is a boundary link with $\mu$ components, then there is an epimorphism $G_L \to F(\mu)$ mapping the set of meridians of $L$ to a free basis.
\end{proposition}

\begin{proof} Let $L=K_1 \cup \cdots \cup K_\mu\subset \Si \times I$, and set $Y=(\Si \times I \sm L)/\Si \times \{1\}.$
Since $L$ is a boundary link, we have a collection $S_1 \cup \cdots \cup S_\mu$ of pairwise disjoint surfaces in $\Si \times I$ such that $\partial S_i =K_i$ for $1 \leq i \leq \mu$. Let $N_1,\ldots, N_\mu$ be disjoint open tubular neighborhoods of $S_1,\ldots, S_\mu$ in $\Si \times I$. Since $S_i$ is homotopic to a 1-dimensional complex, it follows that $N_i \approx S_i \times (-1,1)$ for $1 \leq i \leq \mu$. 

Let $B = \vee^\mu_{i=1} S^1$ be a bouquet of $\mu$ circles, with basepoint $*$ the wedge point. Define $p\co Y \to B$ as follows. For $(x,t) \in S_j \times (-1,1) \approx N_j$, let $p(x,t)=e^{ \pi (t+1)\sqrt{-1}} \in S^1_j$,  the $j$-th circle of $B$. This defines $p$ on the union $\cup_{i=1}^\mu N_i$, and we extend $p$ to all of $Y$ by sending points in $Y \sm \cup_{i=1}^\mu N_i$ to the  basepoint $*\in B$.
Obviously, the bouquet of circles is a $K(F(\mu),1)$, and $p_*\co \pi_1(Y,*) \lto \pi_1(B,*)$ maps the meridians of $L$ to a set that freely generates $\pi_1(B) \cong F(\mu)$. Since $\pi_1(Y) \cong G_L$, we conclude that $p_*$ gives an epimorphism $G_L \to F(\mu)$ sending the meridians of $L$ to a free basis for $F$.  
\end{proof}

Given a finitely generated group $G$, let
$G'=[G,G]$ and $G''=[G',G']$ denote the first and second commutator subgroups of $G$.

The next result is an immediate consequence of  Proposition \ref{prop_epi} and 
Theorem 4.2 in \cite{Hillman-2012}, which asserts that $\cE_{\mu-1}=0$ whenever there is an epimorphism $G_L/G_L''\to F(\mu)/F(\mu)''$.

\begin{corollary} If $L \subset \Si \times I$ is a boundary link with $\mu$ components and link group $G_L$,
then $\cE_{\mu-1}  =0$.
\end{corollary}

\subsection{Lower central series}
For a group $G$, the \emph{lower central series} is the descending series  $G=G_1 \rhd G_2 \rhd \cdots$ of normal subgroups defined inductively by $G_1=G$ and $G_n=[G,G_{n-1}]$ for $n \geq 2$.  Set $G_\om=\bigcap_{n=1}^{\infty} G_n$. The quotient $G/G_n$ is called the $n$-th nilpotent quotient of $G$ and $G_\om$ is called the \emph{nilpotent residual} of $G$. 

The next theorem is a classical result due to Stallings, see \cite[Theorem 3.4]{stallings_central}.

\begin{theorem}[Stallings] \label{stallings} Let $h \co A \to B$ be a homomorphism of groups inducing an isomorphism $H_1(A) \cong H_1(B)$ and a surjection $H_2(A) \to H_2(B)$. Then for finite $n$, $h$ induces an isomorphism $A/A_n \cong B/B_n$. If $h$ is onto, then $h$ also induces an isomorphism $A/A_\om \cong B/B_\om$.  
\end{theorem}

In the following, let $F(\mu)$ be the free group on $\mu$ generators.

\begin{theorem} \label{thm-boundary-link-long}
Suppose $L \subset \Si \times I$ is a boundary link with $\mu$ components. Then $G_L/(G_L)_n \cong F(\mu)/F(\mu)_n$ for all finite $n$, and $G_L/(G_L)_\om \cong F(\mu)/F(\mu)_\om$. Further, every longitude  of $L$ satisfies $\ell_i \in (G_L)_\om$. 
\end{theorem}

\begin{proof} The epimorphism $G_L \to F(\mu)$ of Proposition \ref{prop_epi} clearly induces
an isomorphism $H_1(G_L) \to H_1(F(\mu))$ and a surjection $H_2(G_L) \to H_2(F(\mu))$.
(Indeed, $H_2(F(\mu))=0$.)
Thus, Theorem \ref{stallings}
applies to prove the first part. 

To prove the second part, recall that finitely generated free groups are residually nilpotent, hence $F(\mu)_\om=1$. Thus, showing that the longitudes of $L$ lie in $(G_L)_\om$ 
is equivalent to showing they lie in the kernel of the epimorphism $G_L \to F(\mu)$. 

Recall the construction of the map $p\co (Y,*) \to (B,*)$ of pointed spaces from the proof of Proposition \ref{prop_epi}. (Here $Y = (\Si \times I \sm L) / \Si \times \{1\}$ and
$B =\vee_{i=1}^\mu S^1$.)

Let $K_i$ be the $i$-th component of $L$, and $\ell_i$ be the longitude of $K_i$.
We choose a representative for $\ell_i$ which is a simple closed curve on the Seifert surface $S_i$ for $K_i$. By pushing this curve in the positive normal direction to $S_i$, we can arrange that it lies outside the open tubular neighborhood $N_i$ of $S_i$. By construction, $p$ sends this curve to the basepoint $*$, and it follows that $\ell_i \in \ker(p_*)$. This completes the proof of the second part.
\end{proof}


\subsection{The Chen-Milnor Theorem}
Milnor described a presentation for the nilpotent quotients of the link group of a classical link in \cite{milnor}, and this was extended to virtual links in \cite{dye_kauffman}. We present an alternative proof for virtual links based on the argument of Turaev \cite{Turaev-1976}.

Let $L=K_1 \cup \cdots \cup K_\mu$ be an oriented virtual link diagram  with $\mu$ components. We regard the  link group $G_L$ with its Wirtinger presentation as in Definition \ref{defn:link-group}. For $i=1,\ldots, \mu$, we choose one arc on the $i$-th component $K_i$ of $L$ and let $m_i$ denote the resulting meridian, which is given by the associated generator of \eqref{eqn:link-group}. The corresponding longitude $\ell_i$ is the word obtained by following $K_i$ in the direction of its orientation, starting at the chosen arc, and at each undercrossing recording the overcrossing arc with sign $\ep = \pm 1$ according to the writhe of the crossing. The resulting word is multiplied by the appropriate power of $m_i$ so that the longitude $\ell_i$ has $m_i$-th exponential sum equal to zero.

\begin{theorem} \label{theorem_milnor} 
Let $L \subset \Si \times I$ be a link with $\mu$ components and link group $G_L$,
and let $F(\mu)$ be the free group generated by $x_1,\ldots,x_\mu$.
Then $G_L/(G_L)_n$ has a presentation:
\begin{equation}\label{eqn-milnor}
\left< x_1,\ldots,x_\mu \mid [x_1,\ell_{1,n}],\ldots,[x_\mu,\ell_{\mu,n}], F(\mu)_n\right>,
\end{equation}
where $x_i$ and $\ell_{i,n}$ represent the images in $G_L/(G_L)_n$ of the $i$-th meridian and longitude, respectively.
\end{theorem}
\begin{proof}
Write $L = K_1 \cup \cdots \cup K_\mu$, and let $N(L)= N(K_1)\cup \cdots \cup N(K_\mu)$ be the union of disjoint tubular neighborhoods in $\Si \times I$. Each $N(K_i)$ is a solid torus and has boundary $\partial N(K_i) \approx T^2$.

Let $X= \Si \times I \sm \Int(N(L))$ be the exterior of $L$ in $\Si \times I$, and set $Y=X/\Si \times \{1\}$. Notice that $Y$ is homotopic to $X$ with a cone placed on top, and that
 $G_L \cong \pi_1(Y, *)$. In the following, we choose the basepoint $*$ to lie on $\Si \times \{0\}$.

We choose arcs $\al_i$ in $X$ from $*$ to $\partial N(K_i)$ whose interiors are disjoint and lie in $\Int(X)$. Let $V$ be a closed regular neighborhood of $\cup_{i=1}^\mu \al_i$ in $X$, and notice that $V \approx D^3$.
Further, let $D_i = V \cap \partial N(K_i)$, so that $D_i \approx D^2$. Thus $\overline{\partial N(K_i) \sm D_i}$ is a once-punctured torus.

Let $W = \overline{Y \sm V}$ and $G=\pi_1(W,*).$ Clearly $H_1(W; \ZZ) \cong \ZZ^\mu$, with one generator for each meridian of $L$. We claim that $H_2(W;\ZZ)=0$. To see this, notice that $X \sm V$ is a 3-manifold with two boundary components, one is $\Si \times \{1\}$ and the other is the connected sum of $\Si \times \{0\}$ with the tori $\partial N(K_1), \ldots, \partial N(K_\mu)$ connected along the tubes $\al_i$. Writing $\partial (X\sm V) = \partial_1 \cup \partial_2$ and using the long exact sequence for the pair 
$(X \sm V, \partial_1 \cup \partial_2)$, it follows that $H_2(X \sm V)$ is generated by $\partial_1$ and $\partial_2$ subject to the single relation $\partial_1 + \partial_2 =0.$ Thus $H_2(X \sm V)=\ZZ$, and since $W = (X \sm V)/\Si \times \{1\}$, it follows that $H_2(W;\ZZ)=0$ as claimed.

One can construct a $K(G,1)$ by attaching cells of dimension $d \geq 3$ to $W$, and the induced map $H_2(W) \lto H_2(G)$ is therefore surjective. Thus $H_2(G)=0$, and  
applying  Theorem \ref{stallings}, we see that the inclusion of the meridians induces an isomorphism $F(\mu)/F(\mu)_n \cong G/G_n$ for all $n \geq 1.$ 
Since $Y = W \cup (\cup_{i=1}^\mu D_i) \cup V$, it follows that the link group $G_L$ is isomorphic to the quotient $G/\langle \langle \partial D_i \rangle \rangle,$ where $\partial D_i$ represents the commutator of curves in $W$ whose images in $G_L$ are the meridian and longitude of the $i$-th component. This completes the proof.
\end{proof}

\begin{proposition} \label{prop-nil-quot}
Let $L \subset \Si \times I$ be a link with $\mu$ components and link group $G_L$.
Then $G_L/(G_L)_n \cong F(\mu)/F(\mu)_n$ if and only if all the longitudes of $L$ lie in the nilpotent residual $(G_L)_\om$.
\end{proposition}
\begin{proof} 
If the longitude $\ell_i$ lies in $(G_L)_{n-1}$, then the corresponding element $\ell_{i,n}$ in the presentation
of \eqref{eqn-milnor} necessarily lies in $F(\mu)_{n-1}.$ Consequently, if all the longitudes of $L$ lie in $(G_L)_{n-1}$, then $[x_i,\ell_{i,n}] \in F(\mu)_n$ for $i=1,\ldots, \mu$. In particular, if $\ell_i \in (G_L)_\om \lhd (G_L)_{n-1}$ for $i=1,\ldots, \mu$, then the relations $[x_i,\ell_{i,n}]$  are consequences of the relations $F(\mu)_n=1$.
In that case, the presentation \eqref{eqn-milnor} of Theorem \ref{theorem_milnor} simplifies to show that
$G_L/(G_L)_n \cong F(\mu)/F(\mu)_n,$ which proves one direction.

To prove the other direction, suppose conversely that $G_L/(G_L)_n \cong F(\mu)/F(\mu)_n$ for all $n$. Then the relations $[x_i,\ell_{i,n}]=1$ of the presentation \eqref{eqn-milnor} must be consequences of the relation $F(\mu)_n=1$. Recall that the $i$-th longitude $\ell_i$ is defined to have zero exponential sum with respect to the $i$-th meridian, and so Corollary 5.12 (iii) of \cite{Magnus-Karrass-Solitar} applies to show that since $[x_i,\ell_{i,n}] \in F(\mu)_n$, we must have $\ell_{i,n} \in F(\mu)_{n-1}$. Consequently, it follows that $\ell_i \in (G_L)_{n-1}$, and since this holds for all $n \geq 3,$ we conclude that $\ell_i \in   
(G_L)_\om$.
\end{proof}


\subsection{The Chen groups}
Given a finitely generated group $G$, 
define the \emph{Chen groups} of $G$ to be the quotient groups
$$Q(G;q) = \frac{G_q G''}{G_{q+1}G''}.$$
Let $G(\infty)=\bigcap_{n=1}^{\infty} G_nG''$.

The following result is the main theorem of \cite{hill}, extended slightly to include link groups of virtual links in part (3). 
 
\begin{theorem}[Hillman] \label{hillman} 
If $G$ is a finitely presented group with $G/G' \cong \ZZ^\mu$, then the following conditions are equivalent.
\begin{enumerate}
\item[(1)] $\cE_{\mu-1}=0$ (i.e. the first $\mu-1$ Alexander ideals vanish).
\item[(2)] For each integer $q \ge 1$, the $q$-th Chen group of $G$ satisfies
$$Q(G;q) \cong Q(F(\mu); q),$$
where $F(\mu)$ is a free group of rank $\mu$.
\end{enumerate}
Further, if $G=G_L$ is a link group of a virtual link (see Definition \ref{defn:link-group}), then (1) and (2) are equivalent to:
\begin{enumerate}
\item[(3)] The longitudes of $L$ are in $G(\infty).$
\end{enumerate}
\end{theorem}

\begin{proof}
The proof of equivalence of (1) and (2) can be found in \cite{hill}.
Our focus is on showing that
(2) and (3) are equivalent for link groups of virtual links. To that end, suppose that $G=G_L$ is the link group of a link
$L \subset \Si \times I$. Let $\ell_1,\ldots, \ell_\mu$ denote the longitudes of $L$.

(2) implies (3). This follows by the proof given in \cite{hill}.

(3) implies (2). Suppose that $\ell_i \in G(\infty)$ for $i =1,\ldots, \mu$.
Let $F=F(\mu)$ be the free group generated by $x_1,\ldots,x_\mu$.
Then arguing as in Proposition \ref{prop-nil-quot}, it follows that
$G/G_q G'' \cong F/F_q F''$ for all $q \geq 1$. Consider the following commutative diagram, where the vertical maps $\al, \be$, and $\ga$ are induced by the homomorphism $F \to G$ sending $x_i$ to a meridian for the $i$-th component of $L$.
$$\xymatrix{
1 \ar[r] & {F_qF''}/{F_{q+1} F''} \ar[r] \ar[d]^\al & {F}/{F_{q+1} F''} \ar[r] \ar[d]^\be & {F}/{F_q F''} \ar[r] \ar[d]^\ga & 1 \\
1 \ar[r] & {G_q G''}/{G_{q+1} G''} \ar[r] & {G}/{G_{q+1} G''} \ar[r] & {G}/{G_q G''} \ar[r] & 1
}$$
The two rows are exact, and $\be$ and $\ga$ are isomorphisms. Therefore,  $\al$ is also an isomorphism, and it follows that $Q(F; q) \cong Q(G;q)$ for all $q \geq 1$. Thus, we see that (2) holds.
\end{proof}


\section{Crossing the crossings and the tube construction} \label{S4}
\subsection{The  Bar-Natan \!\! $\Zh$-construction} \label{S4-1}

In this section, all virtual links are assumed to be oriented, and $\Zh$ refers to the Cyrillic letter and is pronounced \emph{``zh''}. This letter was chosen as a symbol for the operation of ``crossing the crossings'' \cite{Dror}.

Given a virtual link diagram $D$, following \cite{Dror}, one can construct a new diagram $\Zh(D)$ as follows.
If $D$ has $\mu$ components, then $\Zh(D)$ has $\mu +1$ components, it is obtained from $D$ by adding one new component, which is drawn in blue and includes two new overcrossings for each classical crossing of $D$ (see Figure \ref{zh-crossing}). The new overcrossings are connected to one another, adding virtual crossings wherever needed to cross any existing arcs.

For a given virtual link diagram $D$ we write $\Zh(D) = D \cup \om.$ Here $\om$ refers to the new component and $\Zh(D)$ is considered up to \emph{semi-welded equivalence}, which is defined next. 

\begin{figure}[h]
\centering
  \includegraphics[scale=0.70]{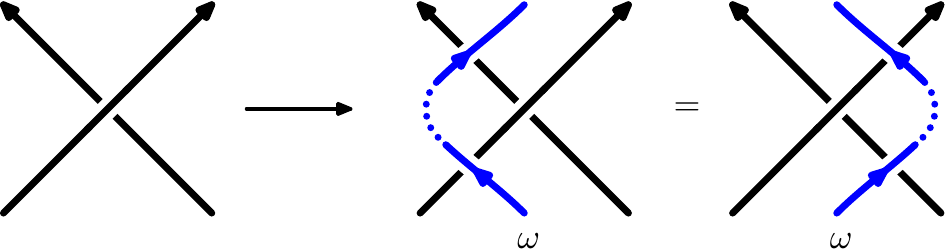} \qquad \qquad  \includegraphics[scale=0.70]{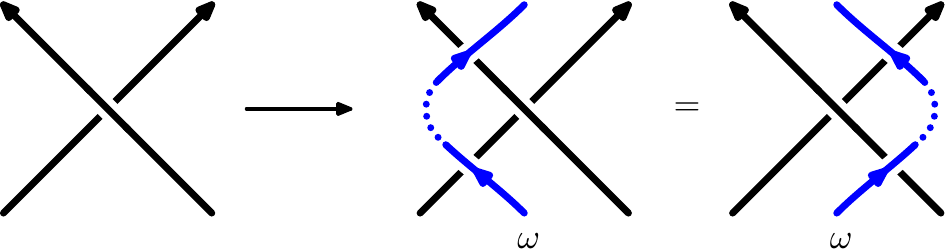} 
 \caption{The $\Zh$-construction for positive and negative crossings.} \label{zh-crossing}
\end{figure} 

\begin{definition} [semi-welded equivalence] \label{defn:semi-welded}
Let $D$ be a virtual link diagram with $\mu$ components $D_1,  \ldots, D_\mu.$
Fix nonnegative integers $n,m$ with $n+m =\mu$, and regard  the components $D_1, \ldots, D_n$
as regular (or virtual) components and the components $D_{n+1}, \ldots, D_{n+m}$ as $\om$-components (or welded components).
A semi-welded equivalence of $D$ is a sequence of moves which include
generalized Reidemeister moves anywhere, and forbidden overpass moves only when the overcrossing arc belongs to one of the $\om$-components.
\end{definition}

An elementary exercise shows that semi-welded equivalence induces an equivalence relation on virtual link diagrams.   
Further, if $D$ and $D'$ are semi-welded equivalent, then they necessarily have the same number of regular components and the same number of $\om$-components. If $m=0,$ then there are no $\om$-components, and $(\mu,0)$ semi-welded equivalence is just virtual equivalence. On the other hand, if $n=0$, then every component is an $\om$-component, and $(0,\mu)$ semi-welded equivalence is just welded equivalence. 
In general, for any $n,m$ with $n+m = \mu$, virtual equivalence implies $(n,m)$ semi-welded equivalence,
and $(n,m)$ semi-welded equivalence implies welded equivalence.

In this paper, we will consider only semi-welded equivalence in the case where there is just one $\om$-component, namely the one arising from the $\Zh$-construction.
Henceforth, we will write ``semi-welded equivalence'' instead of ``$(n,1)$ semi-welded equivalence.''

One immediate consequence is that deleting the $\om$-component from $\Zh(D)$ simply returns the original link diagram $D$.
Furthermore, since the $\om$-component in $\Zh(D)$ consists exclusively of overcrossings, these crossings may be ordered without altering its semi-welded type. This follows essentially from the forbidden overpass move, as we now explain.

First, fix a basepoint and form a loop with the $\om$-component by connecting the overcrossing arcs in any fashion, introducing virtual crossings as needed to make the connections. This provides a way to order the crossings of the $\om$-component, albeit arbitrary. However, note that any two consecutive overcrossings of the $\om$-component can be transposed. If the two undercrossing arcs are adjacent to one another, this can been seen by performing a virtual Reidemeister two move and forbidden overpass move (see Figure \ref{over-crossing}). If the two undercrossing arcs are not adjacent, one can apply detour moves to arrange for them to be adjacent (see Figure \ref{over-crossing-v}).

\begin{figure}[h]
\centering
  \includegraphics[scale=1.0]{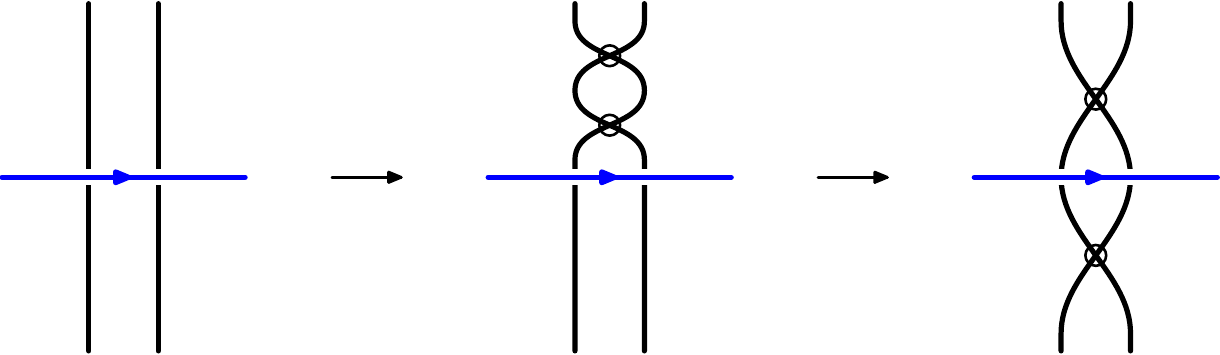} 
 \caption{Two adjacent overcrossings on the $\om$-component can be transposed using a virtual Reidemeister two move and a forbidden overpass move.} \label{over-crossing}
\end{figure} 

\begin{figure}[h]
\centering
\includegraphics[scale=1.0]{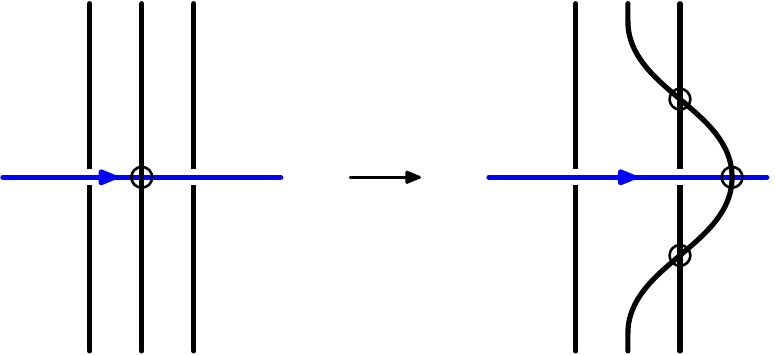} 
 \caption{Detour moves can be used to make any two consecutive overcrossings on the $\om$-component adjacent.} \label{over-crossing-v}
\end{figure} 

Consequently, the semi-welded link type of $\Zh(D)$ is independent of the order chosen for the overcrossings of the $\om$-component. It is also independent of the choice of basepoint, which can be seen using detour moves. In Figure \ref{zh-crossing}, the dotted curves illustrate a way to connect the overcrossings, but one could alternatively connect them in any other way without altering the underlying semi-welded link. 

For any subset of crossings of the $\om$-component, one can draw a loop starting and ending at the basepoint. In that case, we say that the $\om$-component contains the loop.

\begin{proposition}[Bar-Natan] \label{prop:bar-natan}
The semi-welded link $\Zh(D)$ depends only on the virtual link of $D$.
\end{proposition}
\begin{proof}
Suppose $D$ and $D'$ are two virtual link diagrams that are related by generalized Reidemeister moves, then we must show that $\Zh(D)$ and
$\Zh(D')$ are semi-welded equivalent. It is enough to prove this in the case $D$ and $D'$ are related by a single Reidemeister move. 

\begin{figure}[h]
\centering
 \includegraphics[scale=0.85]{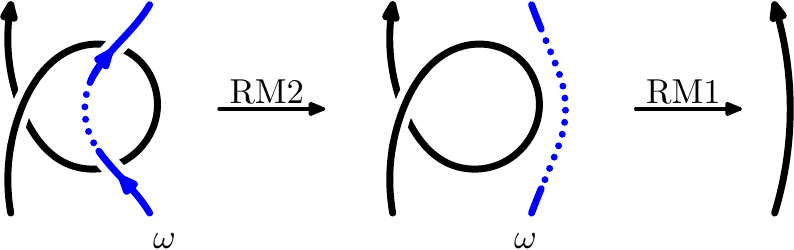} 
 \caption{Invariance of $\Zh(D)$ under $\Om_{1a}$.}\label{rm1a-zh}
\end{figure} 

To this end, Polyak's result in \cite{Polyak} shows that all the Reidemeister moves can be generated from the four moves $\Om_{1a},\Om_{1b},\Om_{2a},$ and $\Om_{3a}$, so it is enough to prove invariance of $\Zh(D)$ under each one of these four moves.

\begin{figure}[h]
\centering
 \includegraphics[scale=0.85]{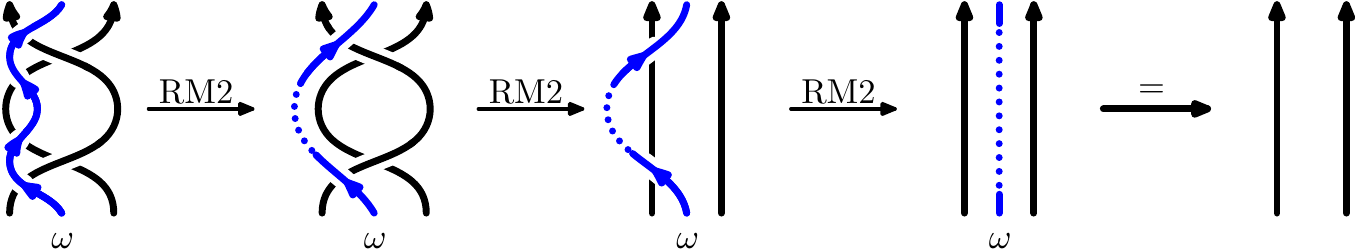} 
 \caption{Invariance of $\Zh(D)$ under $\Om_{2a}$.}\label{rm2a-zh}
\end{figure} 

The proof in the case of the $\Om_{1a}$ move is shown in Figure \ref{rm1a-zh}. Notice that after performing a Reidemeister 2 move, there are no classical crossings on the $\om$-arc, which is why it appears dotted in the middle frame of Figure \ref{rm1a-zh}. As such, it can be erased altogether. The proof for the $\Om_{1b}$ move is similar. 

The argument for the $\Om_{2a}$ move is shown in Figure \ref{rm2a-zh}. The first frame shows the $\om$-arc with four overcrossings, but they all eventually cancel out under successive Reidemeister 2 moves. Once the $\om$-arc can be drawn without classical crossings, it can be erased just as before.  

\begin{figure}[h]
\centering
 \includegraphics[scale=0.85]{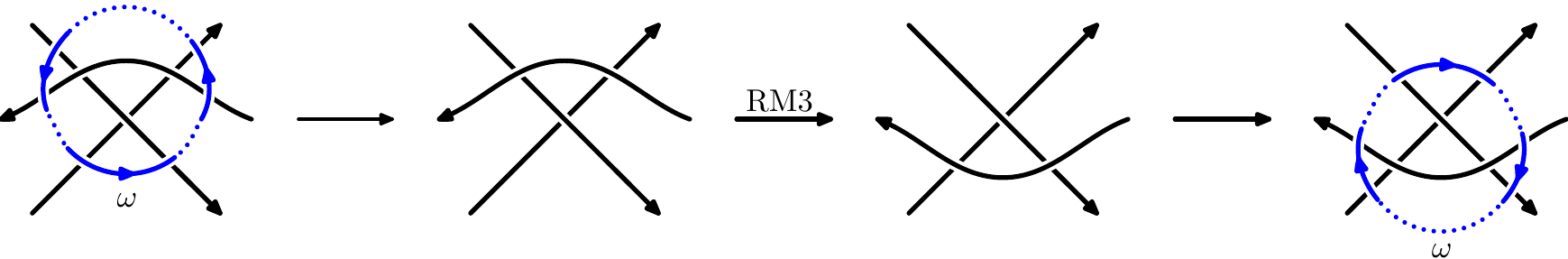} 
  \caption{Invariance of $\Zh(D)$ under $\Om_{3a}$.}\label{rm3a-zh}
\end{figure} 

The proof for the $\Om_{3a}$ move is described in Figure \ref{rm3a-zh}.
In the first frame, the $\om$-arc is seen to contain a loop around the three crossings. 
Notice that this loop can be drawn without any virtual crossings. Therefore, using the forbidden overpass, the loop can be pulled off. After performing a 
Reidemeister 3 move, one can then use the forbidden overpass
to insert a loop going in the opposite direction. This shows invariance of
the $\Zh$-construction under the $\Om_{3a}$ move and completes the proof.
\end{proof}

Figure \ref{2-1-Zh} shows the result of applying the $\Zh$-construction to the virtual trefoil, also known as the virtual knot $2.1$, and in that figure all crossings along the dotted blue curves are virtual.

\begin{figure}[h]
\centering
\begin{tabular}{c} 
$$\xymatrix{ \begin{tabular}{c} \includegraphics[scale=0.950]{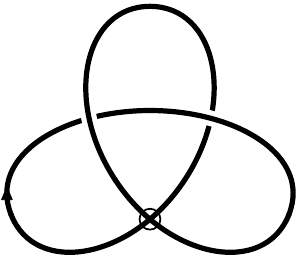} \end{tabular} \ar[r]^{\zh}  
& \begin{tabular}{c}\includegraphics[scale=0.950]{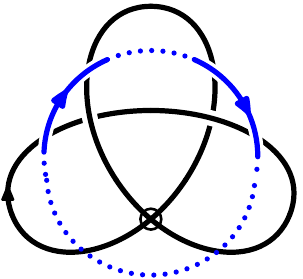} \end{tabular} }$$ 
\end{tabular}
 \caption{The $\Zh$-construction applied to the virtual trefoil.}\label{2-1-Zh}
\end{figure} 

We now describe an alternative formulation of the $\Zh$-construction in terms of Gauss diagrams which is sometimes useful. Suppose that $\Gamma$ is a Gauss diagram representing the virtual link $L$.
Then $\Zh(\Gamma)$ is the Gauss diagram with one extra component obtained by adding two new arrows (both overcrossings) to either side of each existing chord of $\Gamma$. The new arrows are drawn in blue, and they comprise the $\om$-component. There are two possible configurations for the new arrows, as shown on the right of Figure \ref{GD-crossing-zh}, either before the foot and after the head of the existing chord, or after the foot and before the head. These two configurations are equivalent by a semi-welded RM3 move. The two new arrowheads have opposite sign; the one nearest the head of the existing chord has sign equal to that of the chord, and the one nearest the foot has the opposite sign. Since $\Zh(\Gamma)$ is considered up to semi-welded equivalence, the order of the new arrows is arbitrary. For that reason, we do not bother to draw the core circle for the $\om$-component. If one wants, one could degenerate the core circle of the  new component to a point, reflecting the fact that  the new arrows can be reordered arbitrarily.

One can use this approach to give an alternative proof of Proposition \ref{prop:bar-natan} in terms of Gauss diagrams. We leave the details as an entertaining exercise for the reader.

\begin{figure}[htb]
\begin{tabular}{c} 
$$\xymatrix{ \begin{tabular}{c} \includegraphics[scale=0.80]{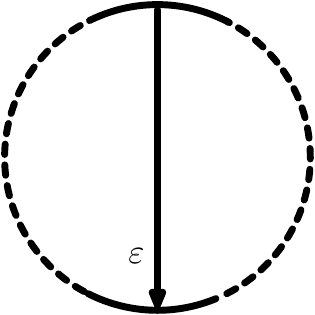} \end{tabular} \ar[r]^{\hspace{-2.0cm}\zh} 
& \begin{tabular}{c}\includegraphics[scale=0.80]{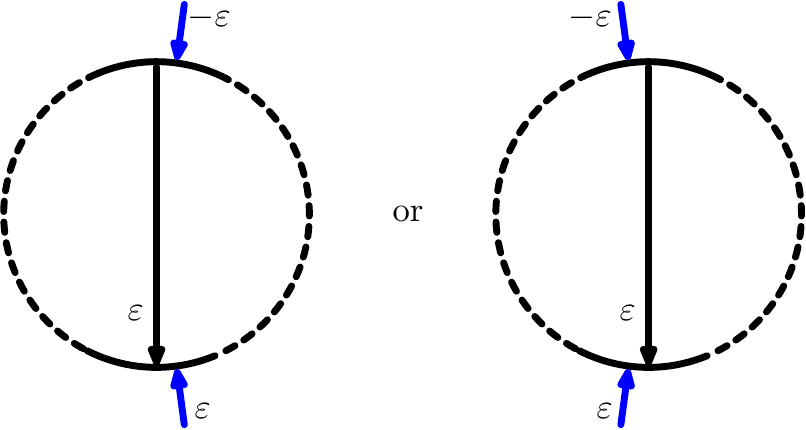} \end{tabular} }$$ 
\end{tabular}
\caption{The $\Zh$-construction applied to a chord in a Gauss diagram.} \label{GD-crossing-zh}
\end{figure}

\begin{proposition}
The $\Zh$-construction is functorial under virtual link cobordism. 
Namely, given a cobordism between virtual link diagrams $D$ and $D'$,
there is a cobordism between the semi-welded link diagrams $\Zh(D)$ and $\Zh(D').$
Furthermore, if the cobordism from $D$ to $D'$ has $b$ births, $d$ deaths, and $s$ saddles,
then the cobordism from $\Zh(D)$ to $\Zh(D')$ also has $b$ births, $d$ deaths, and $s$ saddles.
\end{proposition}
\begin{proof}
Notice that a cobordism can be decomposed into a sequence of operations, each of which is one of the generalized Reidemeister moves or a birth, death, or saddle. We have already seen that the $\Zh$-construction commutes with each of the generalized Reidemeister moves, so all that remains is to prove it commutes with births, deaths, and saddles.

Since the $\Zh$-construction only affects the virtual link locally at the classical crossings, it follows that if $D'$ is obtained from $D$ by a birth, death, or saddle, then $
\Zh(D')$ is obtained from $\Zh(D)$ by the very same birth, death, or saddle.

It follows that if $W$ is a virtual link cobordism from $D$ to $D'$, then
$\Zh(W)$ is a semi-welded link cobordism from $\Zh(D)$ to $\Zh(D').$ 
\end{proof}

Observe that, from the above proof, it follows that the Euler number of the cobordism from $D$ to $D'$ is equal to that of the cobordism from $\Zh(D)$ to $\Zh(D')$.

\begin{corollary} \label{cor_conc}
If $L$ and $L'$ are  concordant virtual links, then $\Zh(L)$ and $\Zh(L')$ are concordant as semi-welded links.
\end{corollary}


In  \cite{Dror}, Bar-Natan proves that the $\Zh$-map ``factors on $u$-links,'' which is to say that if $L$ is a classical link, then the $\om$-component is unknotted and unlinked. The next theorem shows that $\Zh(L)$ factors similarly for almost classical links. 

\begin{theorem} \label{thm:ACsplit}
If $L$ is an almost classical virtual link, then the $\om$-component is unknotted and unlinked and $\Zh(L)$ 
 is semi-welded equivalent to the split almost classical link $L \sqcup \om$. In particular, if $L$ is a virtual boundary link, then $\Zh(L)$ is semi-welded equivalent to a virtual boundary link.
\end{theorem}
\begin{proof} In the proof of invariance under $\Om_{3a}$ in Proposition \ref{prop:bar-natan},
we showed that the $\om$-arc contains a loop without virtual crossings which can be pulled off. The same reasoning applies to any disk-like region bounded by $n$ arcs of a link diagram provided they are all oriented in the same direction. 

This observation is a key step here, and to use it we claim that $L$ can be represented by a special kind of diagram $D$ on $\Si$. A diagram $D$ on $\Si$ is called \emph{checkerboard colorable} if the regions of $\Si \sm D$ can be colored black or white so that adjacent regions have opposite colors. One may further assume, by destabilizing if necessary, that $\Si \sm D$ is a union of disks.

A virtual link is almost classical if and only if it bounds a Seifert surface, and every almost classical link is checkerboard colorable (see \cite{acpaper}).
Given such a link, using isotopies, one can arrange that the black surface of the checkerboard coloring is a Seifert surface. This will be the case if and only if all crossing are type I crossings (see e.g.~\cite{Burde-Zieschang}), which means that the black regions, which are all disks, are the ones obtained by performing oriented smoothings at all crossings. 

We therefore assume that $L$ is given as a checkerboard colored diagram where the black regions form a Seifert surface. Each black region is a disk and is oriented so that its boundary orientation agrees with the orientations of the arcs of the diagram that form its boundary.  The crossings of the $\om$-arc are co-oriented with the arcs of the diagram (see Figure \ref{fig:zh-shaded}), and thus the crossings of the $\om$-arc around a black region induce the same orientation on it. 

\begin{figure}[h]
\centering
 \includegraphics[scale=1.0]{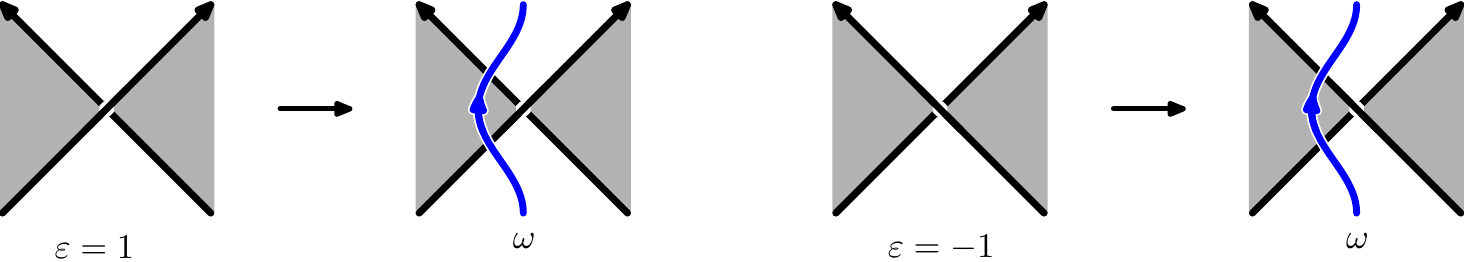} 
 \caption{The $\om$-arc for a checkerboard colored link at a positive and negative crossing.}\label{fig:zh-shaded}
\end{figure} 

The black regions are partitioned into two subsets. One consists of the black regions with boundary oriented counterclockwise, and the other consists of the black regions whose boundary is oriented clockwise. (This reflects the well-known fact that the Tait graph of a Seifert surface is bipartite.) The following argument would work with either subset of black regions, but we phrase it in terms of counterclockwise oriented regions, which coincide with the regions to the left of the $\om$-arcs. The crossings of the $\om$-arc around any such region may be stitched together to form a loop. These loops bound disks, namely the black regions themselves. We claim that these loops can be pulled off.

To prove the claim, fix a black disk and project the surface onto the plane so that the disk is embedded. This can be achieved by shrinking the disk by an isotopy if necessary. The projection gives a virtual link diagram for $L$ which has no virtual crossings in the image of the disk. The $\om$-arc will then contain a loop going around the image disk, and notice that this loop will not contain any virtual crossings. Therefore,  one can pull the loop off the disk, thereby eliminating its classical crossings.
Repeat this argument for the other counterclockwise black regions. By the forbidden overpass move, the loops may be connected in any order to form the $\omega$-component, where only virtual crossings are added between the loops.  By successively pulling loops off their black discs, one can eliminate all the remaining classical crossings of the $\om$-arc.  
This shows that  $L \sqcup \om$ is split and the theorem follows.
\end{proof}


For the next result, recall the definition of the reduced virtual link group $\wbar{G}_L$ (see Definition \ref{defn-red-virt-gp}).
  
\begin{proposition} \label{prop_gp}
Let $L$ be a virtual link with $\mu$ components, and let $\Zh(L)$ be the semi-welded link with
$\mu+1$ components obtained from the $\Zh$-construction.  Then there is an isomorphism of groups
$$\wbar{G}_L \cong G_{\! \zh(L)}$$
from the reduced virtual link group of $L$ to the link group of $\Zh(L)$.
\end{proposition}

\begin{proof}
The reduced virtual link group $\wbar{G}_L$ is generated by the short arcs of $L$, which run from one classical crossing of $L$ to the next. The link group $G_{\zh(L)}$ is generated by the arcs of $\Zh(L)$, which run from one undercrossing of $\Zh(L)$ to the next.
In Figure \ref{zh-group-rel}, we consider the effect of the $\Zh$-construction on the crossings of $L$. The $\om$-component divides each overcrossing arc into two. Consequently, the $\Zh$-construction divides the arcs of $L$ into short arcs. It also introduces one new arc, denoted $x$ in Figure \ref{zh-group-rel}, along the outgoing undercrossing arc.  The relations for $G_{\zh(L)}$ are given in Figure \ref{zh-group-rel}. Since the generator $x$ can be written as a word in the other generators, it can be removed from the presentation of $G_{\zh(L)}$. The remaining generators and relations match exactly those of $\wbar{G}_L$. Indeed, for a positive crossing, substituting $x=c^{-1} b c$ and $c=\om a \om^{-1}$ into the bottom equation gives:
$$d=\om^{-1} x \om=\om^{-1}c^{-1} b c \om= \om^{-1} \om a^{-1} \om^{-1} b \om a \om^{-1} \om=a^{-1}\om^{-1} b \om a.$$
For a negative crossing, a similar argument gives $c=b \om a \om^{-1} b^{-1}$. 

Define $\Phi \co \wbar{G}_L \lto G_{\zh(L)}$ to be the homomorphism that sends the short arc generators of $L$ to the corresponding arc generators of $\Zh(L)$ and sends $v$ to $\om.$ 
Comparing the relations of Figure \ref{zh-group-rel} with the defining relations of $\wbar{G}_L$ (see Figure 
\ref{crossing-red}), it follows that $\Phi$ is an isomorphism.
\end{proof}

\begin{figure}[h]
\centering
  \includegraphics[scale=0.90]{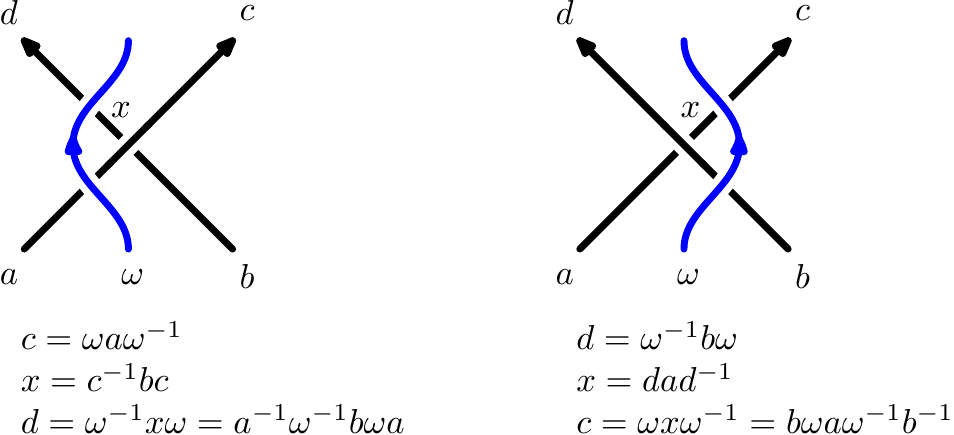}  
 \caption{The relations in the link group of $\Zh(L)$ at a crossing.} \label{zh-group-rel}
\end{figure} 

Considering $\Zh(L)$ as a welded link, since the $\om$-component has only overcrossings, one can easily see that the Wirtinger presentation of the link group $G_{\zh(L)}$ in Equation \eqref{eqn:link-group} has deficiency one. 

If $L$ is almost classical, then Theorem \ref{thm:ACsplit} applies to show that  the meridian of the $\om$-component generates a free cyclic summand of the link group $G_{\zh(L)}$. Consequently, we have $G_{\zh(L)}\cong G_L*\ZZ$.  Combined with Proposition \ref{prop_gp}, this provides a new proof of Theorem 5.2 and Corollary 5.4 of \cite{acpaper}.

\begin{corollary} If $L$ is almost classical, then $\wbar{G}_L \cong G_L * \ZZ$ and $\De^0_L(s,t)=0.$
\end{corollary}

In general, for virtual links $L$ the isomorphism $\Phi \co \wbar{G}_L \to G_{\zh(L)}$ of Proposition \ref{prop_gp} commutes with the abelianization maps $\wbar{G}_L \to \ZZ^{\mu +1}$ and  $G_{\zh(L)} \to \ZZ^{\mu +1}.$
Consequently, the Alexander invariants of $\wbar{G}_{L}$ and $G_{\zh(L)}$ are isomorphic.
In particular, for a virtual knot $K$, the generalized Alexander polynomial of $K$ vanishes if and only if the first elementary ideal $\cE_1$ of $G_{\, \zh(K)}$ vanishes. (Recall from \S \ref{subsec:sketch} the degree shift  associated with the $\Zh$-map.)


\subsection{Concordance and the Tube map}  \label{S4-2}
In \cite{Satoh}, Satoh defined the Tube map, which associates a ribbon torus link $\Tube(L)$ in $S^4$ to
any virtual link $L$. In \cite[\S 3.1.1.]{Bar-Natan-Dancso-2016}, readers will find a purely topological description of the Tube map for framed welded links. More details about the Tube map can also be found in the article \cite{Audoux-2016}.

\begin{figure}[htb]
\begin{tabular}{cc} \begin{tabular}{c}
\def\svgwidth{1.5in}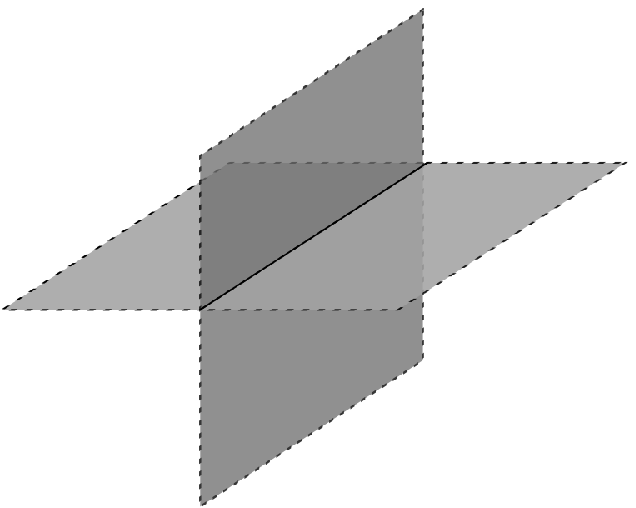 \end{tabular} & \begin{tabular}{c}
\def\svgwidth{1.85in}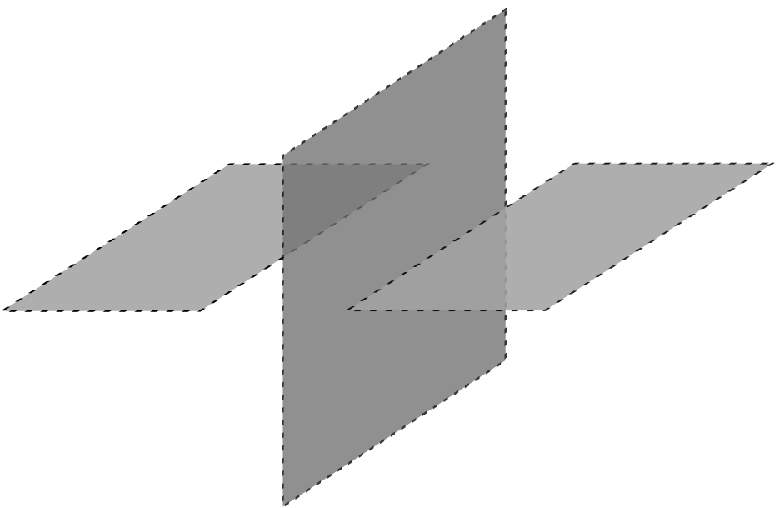 \end{tabular}  \end{tabular}
\caption{Local picture of a double line of the immersed surface (left) and a broken surface diagram (right).} \label{fig_broken}
\end{figure}

We describe the Tube map using broken surface diagrams, and to that end we introduce some terminology. 
A \emph{knotted surface} is an embedding of a surface into $\RR^4$; and it is called \emph{ribbon} if it bounds an immersed handlebody in $\RR^4$ with only ribbon singularities (see \cite[Definition 2.4]{Carter-Kamada-Saito-2004}).
Every knotted surface can be represented as a \emph{broken surface diagram}, and vice versa \cite[Proposition 1.21]{Carter-Kamada-Saito-2004}. A broken surface diagram depicts the generic surface obtained by taking the image of the surface under  $p\colon \RR^4 \to \RR^3$, where $p$ is projection onto the first three coordinates. On a double line, the generic surface is drawn as broken along the sheet with smaller $x_4$ coordinate (see Figure \ref{fig_broken}).  Two knotted surfaces are equivalent if they are ambient isotopic in $S^4$, and two broken surface diagrams represent equivalent knotted surfaces if and only if they are related by ambient isotopy in $\RR^3$ and a finite sequence of \emph{Roseman moves} (see \cite[Theorem 1.37]{Carter-Kamada-Saito-2004}). 

\begin{figure}[htb]
\begin{tabular}{cc} 
$$\xymatrix{\begin{tabular}{c}\psfig{file=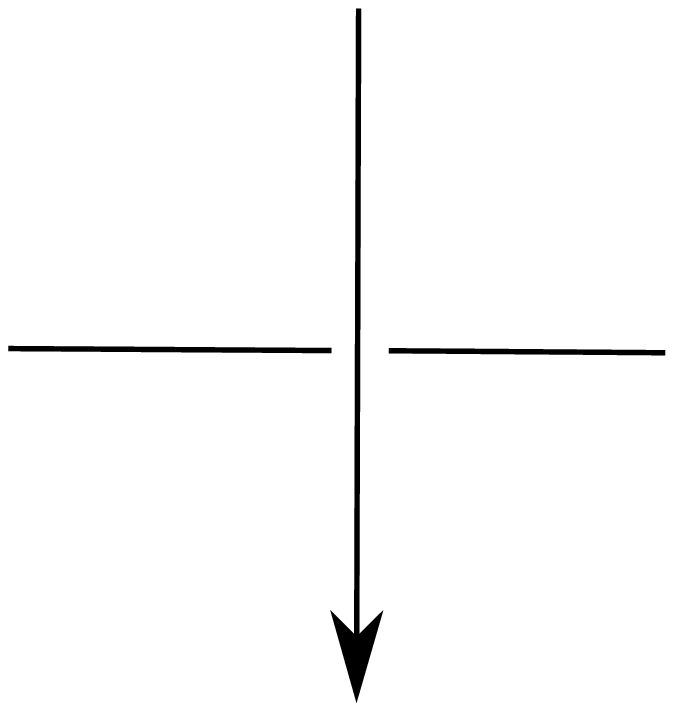,width=0.8in}  \end{tabular} \ar[r]^{\hspace{-.15in}\text{Tube}} &  \begin{tabular}{c}\psfig{file=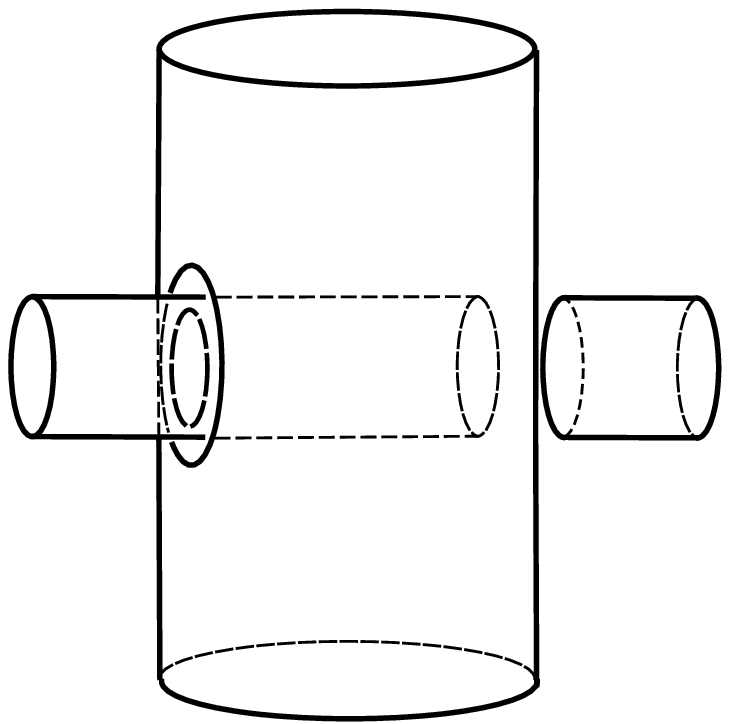,width=1.0in} \end{tabular}}$$ \qquad&
$$\xymatrix{
\begin{tabular}{c}\psfig{file=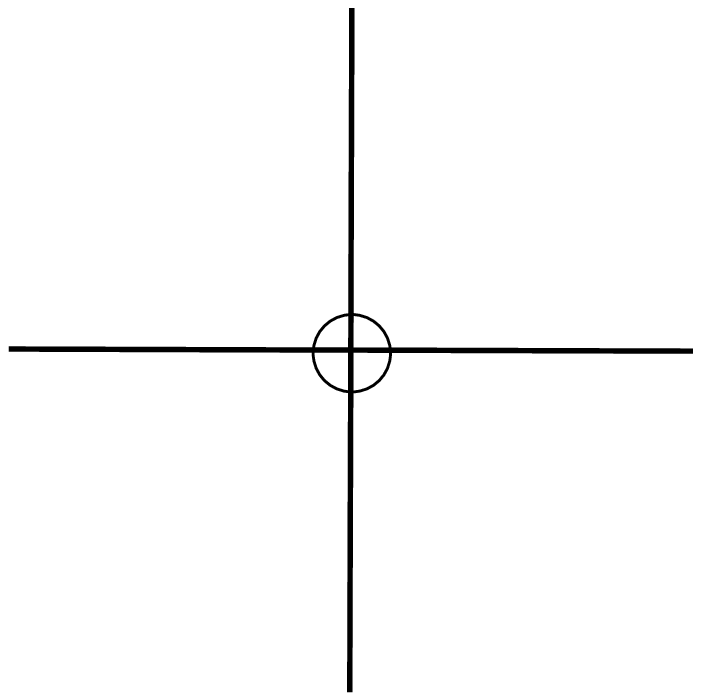,width=0.8in} \end{tabular} \ar[r]^{\hspace{-.15in}\text{Tube}} & \begin{tabular}{c}\psfig{file=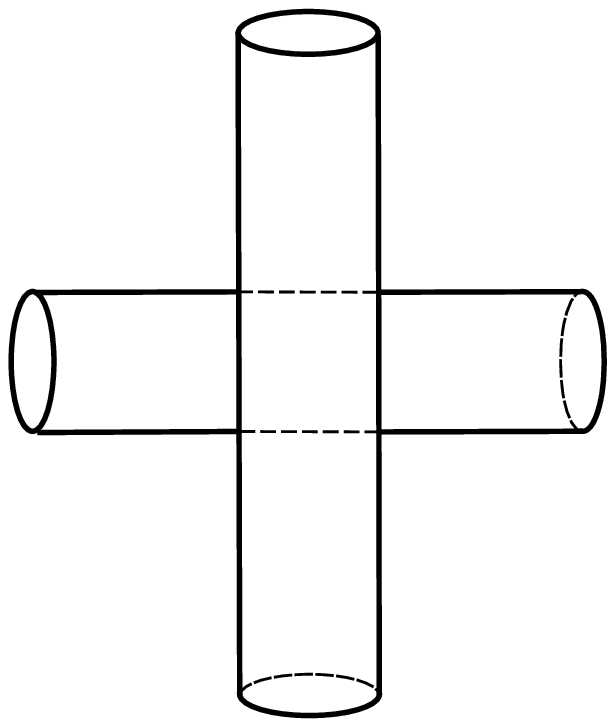,width=1.0in} \end{tabular}}$$
\end{tabular}
\caption{The Tube map.} \label{fig_tube}
\end{figure}

Satoh's Tube map can be described as a broken surface diagram as follows. Each arc is inflated to an annulus or ``tube''.  
At classical crossings, the tube of the overcrossing arc goes inside the tube of the undercrossing arc (see Figure \ref{fig_tube} left), and at virtual crossings, the tubes pass over each other without intersecting (see Figure \ref{fig_tube} right). In \cite{Satoh}, Satoh proved that the equivalence class of $\Tube(L)$ depends only on  $L$ up to welded equivalence.
He further showed that under the correspondence $L \mapsto \Tube(L)$, the knot group $G_L$ and $\pi_1(S^4 \sm \Tube(L))$ are isomorphic. 
 
A \emph{concordance} between two torus links $T_0$ and $T_1$ in $S^4$ is a smooth proper embedding $W$ in
$S^4 \times I$ with finitely many components, each diffeomorphic to $(S^1 \times S^1) \times I,$ 
and such that $W \cap (S^4 \times \{i\}) = T_i$ for $i=0,1.$ A torus link is \emph{slice} if it is concordant to the trivial unknotted unlinked torus link. Note that if two torus links are concordant, then
they have the same number of components.

\begin{proposition} \label{prop_tubes_conc} 
If two virtual links $L_0,L_1$ are welded concordant, then $\Tube(L_0)$ and $\Tube(L_1)$ are concordant ribbon torus links.
\end{proposition}
\begin{proof} Since $L_0$ and $L_1$ are welded concordant, there is a sequence of virtual link diagrams $L_0=J_0,J_1,\ldots,J_n=L_1$ such that for $0 \le i \le n-1$, $J_{i+1}$ is obtained from $J_i$ by either a welded equivalence, a birth, a death, or a saddle. We construct a cobordism $W_i \subset S^4 \times [i,i+1]$ between $\Tube(J_i)$ and $\Tube(J_{i+1})$ for $0 \le i \le n-1$. Then $W=\bigcup_{i=0}^{n-1} W_i \subset S^4 \times [0,n]$ is a cobordism between $\Tube(L_0)$ and $\Tube(L_1)$. Each $W_i$ will be diffeomorphic to $S^1 \times F_i$ where $F_i$ is a surface and $\partial (S^1 \times F_i)=S^1 \times \partial F_i=\Tube(J_{i+1}) \sqcup \Tube(-J_i)$. This implies that $W=S^1 \times \bigcup_{i=0}^{n-1} F_i$. The condition on the Euler characteristic of the welded concordance implies that $F=\bigcup_{i=0}^{n-1} F_i$ is diffeomorphic to a disjoint union of annuli. Hence, $W$ is in fact a concordance between $\Tube(L_0)$ and $\Tube(L_1)$. Consider each type of pair $J_i,J_{i+1}$.

Suppose that $J_{i+1}$ is obtained from $J_i$ by a welded equivalence. Then there is an ambient isotopy $H \colon S^4 \times [0,1] \to S^4$ taking $\Tube(J_i)$ to $\Tube(J_{i+1})$. For $0 \le t \le 1$, consider $H(\Tube(J_i),t)$ as an embedding into $S^4 \times \{i+t\}$ and define $W_i=H(\Tube(J_i),[0,1])\subset S^4 \times [i,i+1]$. Then $W_i$ is diffeomorphic to $\Tube(J_i) \times I$, which is itself a product $S^1 \times F_i$, where $F_i$ is a disjoint union of annuli.

Now consider a birth, death or saddle. In $2$-dimensions, a birth is realized by the addition of a $0$-handle $h_0^2$, a death is realized by the addition of a $2$-handle $h_2^2$, and a saddle is realized by the addition of a $1$-handle $h_1^2$. We claim that $\Tube(J_{i+1})$ can be obtained from $\Tube(J_i)$ by attaching round handles $S^1 \times h_k^2$ to $\Tube(J_i) \times [i,i+1]$ for some $k=0,1,2$. Since $S^1$ has a $1$-dimensional handle decomposition $S^1=h_0^1 \cup h_1^1$, this can be accomplished by appropriately attaching the $3$-dimensional handles $h_0^1 \times h_k^2$ and $h_1^1 \times h_k^2$ (see e.g. \cite[Example 4.6.8]{Gompf-Stipsicz-1999}).

\begin{figure}[t]
\begin{tabular}{ccc} 
\begin{tabular}{c}\def\svgwidth{1in}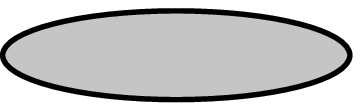 \end{tabular} & \begin{tabular}{c}\def\svgwidth{1.1in}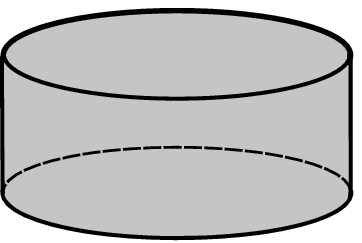 \end{tabular} & 
\begin{tabular}{c}\def\svgwidth{2in}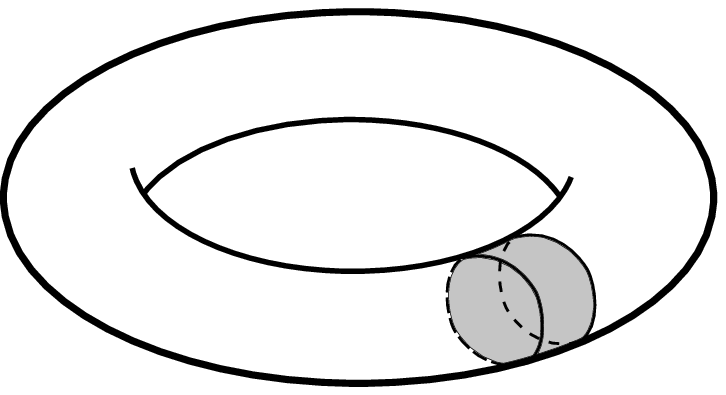\end{tabular} 
\end{tabular}
\caption{A birth is the addition of a $0$-handle $h_0^2$ (left). A round $1$-handle $S^1 \times h_0^1$ is  obtained by attaching a $1$-handle $h_1^1 \times h_0^2$ (right) to a $0$-handle $h_0^1 \times h_0^2$ (center).} \label{fig_zero_handle_attach}
\end{figure}

The attachment for a birth is depicted in Figure \ref{fig_zero_handle_attach}. Note that $(h_0^1 \times h_0^2)\cup (h_1^1 \times h_0^2)$ is a solid torus. Also, $\Tube(J_{i+1})$ is obtained from the broken surface diagram $\Tube(J_i)$ by adding the boundary of an unknotted solid torus. In this case, we may set $W_i=\Tube(J_i)\times [i,i+1] \sqcup S^1 \times h_0^2$. Then $W_i$ is diffeomorphic to $S^1 \times Z \sqcup S^1 \times h_0^2$, where $Z$ is a disjoint union of annuli, and we set $F_i=Z \sqcup h_0^2$. This establishes the claim for a birth. The construction for a death move follows by duality. 

\begin{figure}[htb]
\begin{tabular}{cc} 
\begin{tabular}{c}\def\svgwidth{2.25in}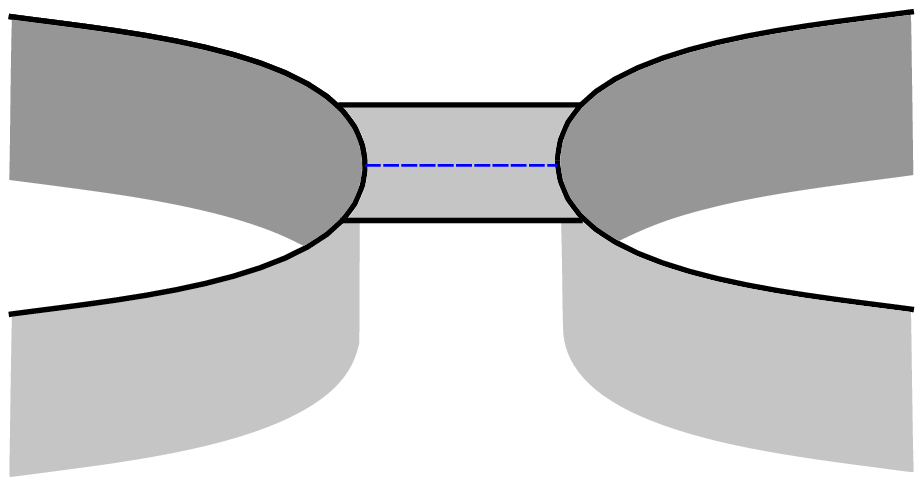 \end{tabular} & \begin{tabular}{c}\def\svgwidth{2.5in}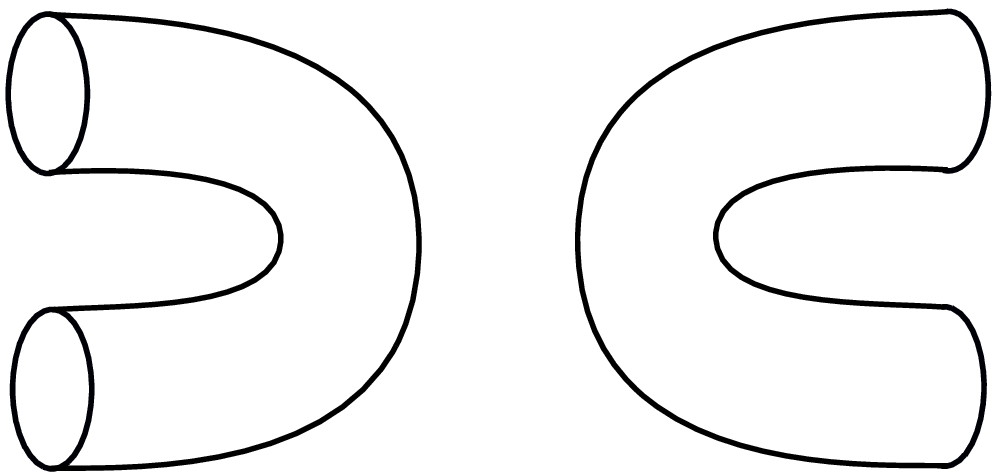 \end{tabular} \\ 
\begin{tabular}{c}\def\svgwidth{2.5in}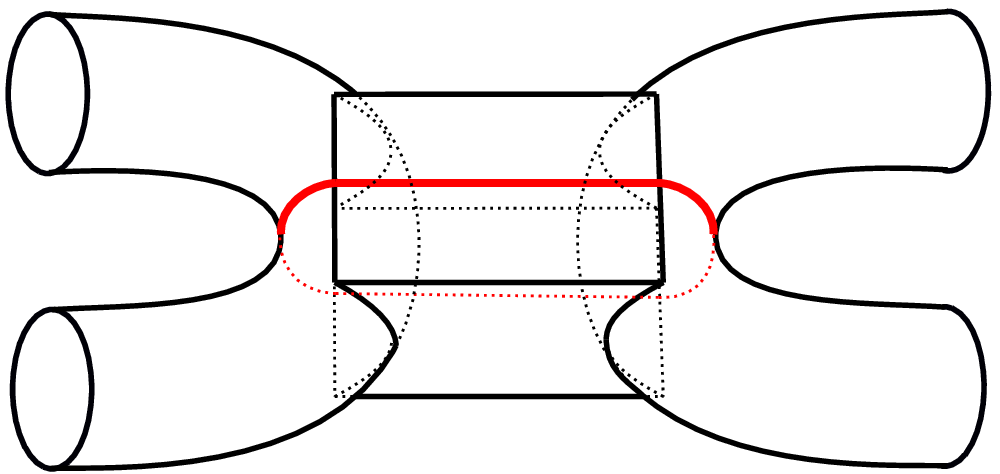 \end{tabular} &
\begin{tabular}{c}\def\svgwidth{2.5in}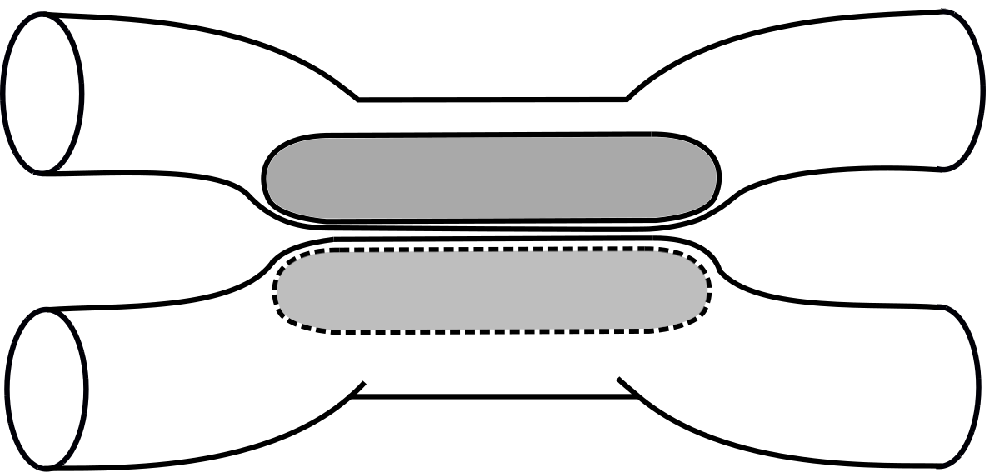 \end{tabular} 
\end{tabular}
\caption{A saddle is the addition of $1$-handle $h_1^2$ (top left). The round 1-handle $S^1 \times h_1^2$ is obtained by attaching a $2$-handle $h_1^1 \times h_1^2$ along the red curve (bottom left) to a $1$-handle $h_0^1 \times h_1^2$.} \label{fig_one_handle_attach}
\end{figure}

The attachment for a saddle is depicted in Figure \ref{fig_one_handle_attach}. Let $Z$ be the disjoint union of annuli satisfying $\Tube(J_i) \times [i,i+1]=S^1 \times Z$. Figure \ref{fig_one_handle_attach}, top left, shows the effect of adding $h_1^2$ to the boundary of $Z$ in $S^4 \times \{i+1\}$. The blue dotted line is the core $C$ of $h_1^2$. Figure \ref{fig_one_handle_attach}, bottom left, shows the effect of adding the $1$-handle $h_0^1 \times h_1^2$ to $S^1 \times \partial Z$. The attaching region for the $2$-handle $h_1^1 \times h_1^2$, drawn in red, is the union of $(S^1\sm\Int(h_0^1)) \times \partial C $ and $(\partial h_0^1) \times C$. The effect of attaching the $2$-handle is to perform a surgery, as in the bottom right of Figure \ref{fig_one_handle_attach}. It follows that $\partial (S^1 \times Z \cup S^1 \times h_1^2)=\Tube(J_{i+1}) \sqcup \Tube(-J_i)$. Set $F_i = Z \cup h_1^2$ and $W_i=S^1 \times F_i$.  
\end{proof}


\subsection{Main result} \label{subsec:main}
The following theorem gives a statement of our main result.

\begin{theorem} \label{thm_main}
Let $K \subset \Si \times I$ be a knot in a thickened surface. If $K$ is virtually concordant to a homologically trivial knot, then its generalized Alexander polynomial vanishes.
\end{theorem}

Restated in terms of virtual knots, we get the following result.
\begin{corollary} \label{cor_main}
If $K$ is a virtual knot that is concordant to an almost classical knot, then $\De^0_K(s,t) =0$. In particular,  
if $K$ is a slice virtual knot, then $\De^0_K(s,t) =0$.
\end{corollary}

\begin{remark}
Theorem  \ref{thm_main} and Corollary \ref{cor_main} hold for all equivalent forms of the generalized Alexander polynomial, including the Sawollek polynomial $Z_K(x,y)$ \cite{saw}, the Kauffman-Radford polynomial $G_K(s,t)$ \cite{Kauffman-Radford}, Manturov's \emph{VA} polynomial \cite{Manturov2002, vktsoa},  the Silver-Williams polynomial $\Delta_0(K)(u,v)$ \cite{Silver-Williams-2003}, and the virtual Alexander polynomial $H_K(s,t,q)$ \cite{Boden-Dies-Gaudreau-2015}.
\end{remark}

A sketch of the proof was given in \S \ref{subsec:sketch}, and below is a detailed argument.
 
Suppose $T_0$ and $T_1$ are two torus links in $S^4$. An \emph{$I$-equivalence} between $T_0$ and $T_1$ is a topological embedding $Z$ in 
$S^4 \times I$ with finitely many components, each homeomorphic to $(S^1 \times S^1) \times I,$ 
and such that $Z \cap S^4 \times \{i\} = T_i$ for $i=0,1.$ Note that if torus links are $I$-equivalent, then they have the same number of components.
Note further that $I$-equivalences are not assumed to be locally flat.

The next result is Theorem 5.2 from \cite{stallings_central}, which
shows that $I$-equivalent compact submanifolds of $S^n$ have isomorphic nilpotent quotients. 

\begin{theorem}[Stallings] \label{thm_stall} 
Let $X$ and $Y$ be $I$-equivalent compact submanifolds of the $n$-sphere $S^n$. Let $A$ and $B$ denote the fundamental groups of $S^n\sm X$ and $S^n\sm Y$, respectively. Then for all $k$, $A/A_k \cong B/B_k$ are isomorphic.
\end{theorem}

Recall that welded concordance of virtual links is given by Definition \ref{defn:welded-concordance}.

\begin{theorem} \label{thm-nil-quot}
If two virtual links $L,L'$ are welded concordant, then $G_{L}$ and $G_{L'}$ have isomorphic nilpotent quotients.
\end{theorem}
\begin{proof}  By Proposition \ref{prop_tubes_conc}, the ribbon torus links $T=\Tube(L)$ and $T'=\Tube(L')$ are concordant. For ribbon torus links, concordance implies $I$-equivalence, hence Theorem \ref{thm_stall} applies to show that the nilpotent quotients of $\pi_1(S^4\sm  T)$ and $\pi_1(S^4\sm  T')$ are isomorphic. Since the link groups satisfy $G_{L}\cong \pi_1(S^4\sm  T)$ and $G_{L'} \cong \pi_1(S^4\sm  T')$, it follows that
$G_L$ and $G_{L'}$ have isomorphic nilpotent quotients.
This completes the proof.
\end{proof}

\begin{theorem} \label{theorem_welded}
Suppose $L \subset \Si \times I$ is a link with $\mu$ components and link group $G_L$. If $L$ is welded concordant to a virtual boundary link, then $\cE_{\mu-1} = 0$.
\end{theorem}
\begin{proof} Theorem \ref{thm-nil-quot} tells us that the nilpotent quotients of $G_L$ are invariants of concordance of welded links. Therefore, if $L$ is welded concordant to a virtual boundary link, then by Theorem \ref{thm-boundary-link-long} and Proposition \ref{prop-nil-quot}, its longitudes must lie in $(G_L)_\om$. Since $(G_L)_\om  \lhd G_L(\infty),$
Theorem \ref{hillman} then applies to show that $\cE_{\mu-1} = 0$.
\end{proof}

\noindent{\it Proof of Theorem \ref{thm_main}.}
Given a knot $K$ in $\Si \times I$ which is virtually concordant to a homologically trivial knot, then it represents 
a virtual knot. By abuse of notation, we use $K$ to denote the resulting virtual knot, which by hypothesis is concordant to an almost classical knot. Set $L = \Zh(K),$ then $L$ is a welded link with $\mu = 2$ components. 
Corollary \ref{cor_conc} implies that $L$ is welded concordant to a virtual boundary link, and Theorem \ref{theorem_welded} applies to show that its first elementary ideal vanishes, i.e., $\cE_1 =0$.
However, Proposition \ref{prop_gp} tells us that the link group $G_L$ is isomorphic to the reduced virtual knot group $\wbar{G}_K$, and thus it follows that $K$ has vanishing generalized Alexander polynomial. \qed

This proof  extends to give the following result for virtual boundary links.
\begin{proposition}
Let $L$ be a virtual link with $\mu$ components. If $L$ is concordant to a virtual boundary link, then the elementary ideals $\wbar{\cE}_\ell$ of the reduced virtual Alexander invariant vanish for $0 \leq \ell \leq \mu-1.$
\end{proposition}


\subsection{Applications}  \label{subsec:applications}
Up to mirror symmetry and orientation reversal, there are 92,800 virtual knots with six or fewer crossings \cite{Green}, and the papers  \cite{Boden-Chrisman-Gaudreau-2017} and
\cite{Rushworth-2019} take up the problem of determining sliceness and the slice genus for these virtual knots.
We describe progress on this problem obtained from Corollary \ref{cor_main}.

With such a large number of virtual knots, the first step is to identify the virtual knots that are possibly slice. Turaev's graded genus $\vartheta(K)$ is defined in \cite{Turaev-2008}, and in
\cite{Boden-Chrisman-Gaudreau-2017} it is shown to be an effective slice obstruction for virtual knots. Indeed, applying $\vartheta(K)=0$ as a sieve reduces the set of potentially slice virtual knots from 92,800 to 1551
(see Table \ref{updated-table}). In combination with the graded genus, the generalized Alexander polynomial eliminates another 194 virtual knots from consideration. Given that 1295 of the resulting virtual knots are known to be slice \cite{Boden-Chrisman-Gaudreau-2017}, this reveals that the combined sieve predicts sliceness for low-crossing virtual knots with 95.3\% accuracy! 

There are a number of other useful concordance invariants of virtual knots, including the Rasmussen invariant derived from Khovanov/Lee homology theories \cite{DKK, Rushworth-2019} and the signature invariants  of almost classical knots \cite{Boden-Chrisman-Gaudreau-2017a}. These invariants not only obstruct sliceness but also provide valuable information on the slice genus. For instance, the signature invariants  were applied to the problem of computing the slice genus of
almost classical knots up to six crossings in \cite{Boden-Chrisman-Gaudreau-2017a}.
Since the generalized Alexander polynomial vanishes on almost classical knots, it is not useful for addressing questions about sliceness for almost classical knots.

\begin{table}[h]  
\begin{tabular}{|c|c|c|c|}
\hline
 Virtual &  & Generalized   & Rasmussen    \\
knot &   Gauss code & Alexander polynomial  & invariant  \\
\hline \hline
4.12 & \small{\tt O1-O2-U1-O3+U2-O4+U3+U4+}& $(1-t)(1-s)(t - s)(1- st)^2$ &$s = 0$ \\ \hline  
5.93  & \small{\tt O1-O2-U1-U2-U3+O4+O3+U5+U4+O5+}&  $-1(1-t)(1-s)(1-st)^3$ &$s = 0$\\ \hline
5.114  & \small{\tt O1-O2-U1-U2-U3+U4-O3+U5+O4-O5+} & $0$ & $s=-2$ \\ \hline
5.212 & \small{\tt O1-O2-U1-O3-U2-O4+U5+U3-O5+U4+}&  $(1-t) (1-s) (1-st)^3$ & $s = 0$ \\ \hline
5.344  & \small{\tt O1-O2+U1-O3-U2+U4+O5+O4+U5+U3-}&  $-(1-s^2)(1-t)^2 (1-st)^2$ & $s = 2$ \\ \hline
5.919  & \small{\tt O1-O2-U1-O3+U4+U2-O5+U3+O4+U5+} &  $1-t)(1-s)(1-st)^3$& $s = 0$ \\ \hline
5.1034  & \small{\tt O1-O2+U1-O3-U4+U3-O5-U2+O4+U5-}&  $-(1-t)  (1-s)  (1-st)^3$ & $s = 0$ \\ \hline
5.1216 & \small{\tt O1-O2+U1-O3-U4+O5-O4+U2+U5-U3-}  &  $0$ & $s = 0$ \\ \hline
5.1963 & \small{\tt O1-O2-O3-U1-U2-U4+O5+U3-O4+U5+}&  $0$ & $s = 0$ \\ \hline
5.2351 & \small{\tt O1-O2-U3+O4+U1-U2-O5-U4+O3+U5-}&  $-(1-t)  (1-s)  (1-st)^3$ & $s = -2$ \\ \hline
5.2430  & \small{\tt O1-U2-O3+U1-O2-U4-O5+U3+O4-U5+}&  $-(1-t^2)(1-s^2)(1-st)^3$ & $s = 0$ \\ \hline
5.2435 & \small{\tt O1-U2-O3-U1-O4+U3-O5+U4+O2-U5+}&  $-(1-t^2) (1-s^2) (1-st)^3$ & $s = 0$ \\ \hline
\end{tabular}
\bigskip
\caption{The generalized Alexander polynomial and Rasmussen invariant. } \label{gen-Alex-table}
\end{table}

The slice genus was determined for all virtual knots with four or fewer crossings with one exception in \cite{Boden-Chrisman-Gaudreau-2017}. In addition, sliceness was determined for all but 11 virtual knots with 5 crossings. These virtual knots are listed in Table \ref{gen-Alex-table},
along with their generalized Alexander polynomials and Rasmussen invariants. The Rasmussen invariants were computed by Will Rushworth in \cite{Rushworth-2019}, and the generalized Alexander polynomials were computed by Lindsay White. Taken together, these computations show non-sliceness for all but two of the twelve virtual knots in Table \ref{gen-Alex-table}.

Thus Corollary \ref{cor_main} implies that the virtual knot 4.12 is not slice, and the results of
\cite{Boden-Chrisman-Gaudreau-2017}
apply to show that it has slice genus $g_s(K)=1$. This
completes the calculation of the slice genus for virtual knots with up to 4 crossings.

For virtual knots with 5 crossings, based on the results of \cite[Table 1]{Boden-Chrisman-Gaudreau-2017} and \cite{Boden-Chrisman-Gaudreau-2017t}, it follows that $5.93, 5.212,$ $5.344,$ $5.919,$ and $5.1034$ all have slice genus $g_s=1$, and that $5.2351, 5.2430,$ and $5.2435,$ have slice genus $g_s = 1$ or 2. There are just two virtual knots with five crossings whose slice status remains unknown, namely $5.1216$ and $5.1963.$
 
According to \cite{Boden-Chrisman-Gaudreau-2017}, there are  236 virtual knots with 6 crossings whose slice status is unknown. Using the signature invariants introduced in \cite{Boden-Chrisman-Gaudreau-2017a}, it follows that the 10 almost classical knots among these are not slice.
Another 39 of these knots were shown to have nontrivial Rasmussen invariants by Rushworth \cite{Rushworth-2019}, therefore none of them are slice either. Of the remaining 187 virtual knots with 6 crossings, exactly 134 have nontrivial generalized Alexander polynomial, and so Corollary \ref{cor_main} implies that none of them are slice. In Example \ref{ex:slice} below, we will show how to slice four of these 6-crossing knots, and that leaves just 39 virtual knots with 6 crossings whose slice status remains unknown. 

\begin{table}[h]  
\begin{tabular}{|c|c|c|c|c|c|c|}
\hline
 Crossing & Virtual & $\vartheta=0$ & combined & slice & status     \\
number & knots &  sieve & sieve & knots &  unknown  \\
\hline \hline
2 & 1 & 0 & 0 &  0& 0\\ \hline
3 & 7 & 1 & 0 &  0& 0\\ \hline
4 & 108 & 15 &  14& 13& 0\\ \hline
5 & 2448 & 59 & 51& 45 & 2\\ \hline
6 & 90235 & 1476 &  1294& 1241 & 39\\ \hline
\end{tabular}
\bigskip
\caption{Virtual knots by crossing number, $\vartheta=0$ sieve, combined $\vartheta=0$ and $\De^0=0$ sieve, slice knots, and status unknown.} \label{updated-table}
\end{table}

\begin{example}  \label{ex:slice}
Among the virtual knots of unknown slice status are 6.33048, 6.33049, 6.33504, 6.33508. Each of them has
vanishing graded genus and generalized Alexander polynomial. Further, each is a connected sum of two virtual knots, one with two crossings and the other with four. We will show that both of these summand virtual knots are slice and conclude the original virtual knot is slice.

\begin{figure}[h]
\centering
\includegraphics[scale=0.75]{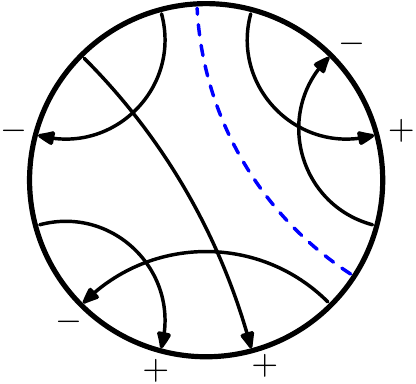} \quad  
\includegraphics[scale=0.75]{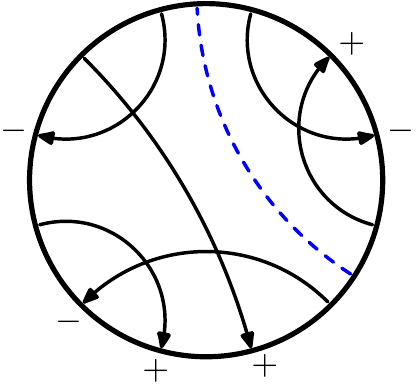} \quad  
\includegraphics[scale=0.75]{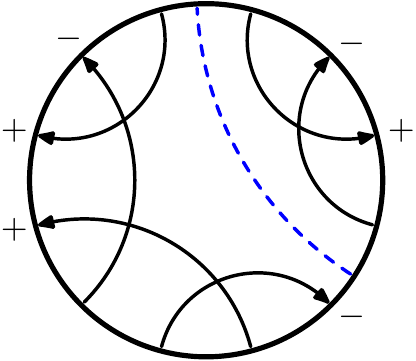} \quad  
\includegraphics[scale=0.75]{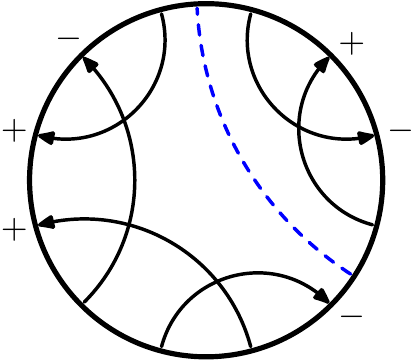}  
 \caption{Gauss diagrams of $6.33048, 6.33049, 6.33504,$ and $6.33508$, viewed as connected sums.}\label{fig:GD-6}
\end{figure} 

The component knot with two crossings is a virtual unknot diagram, therefore it is slice. The other component is a virtual knot $K$ with  four crossings, and it has Kauffman bracket polynomial $\langle K \rangle =1$ and graded genus $\vartheta(K)=0$.  Therefore, Theorem 5.1 of \cite{Boden-Chrisman-Gaudreau-2017} applies to show that it is also slice, and this implies that the original 6-crossing virtual knot is slice.
\end{example}

As a final application, we show that the writhe polynomial is a concordance invariant for virtual knots, giving an alternative proof of \cite[Theorem 3.3]{Boden-Chrisman-Gaudreau-2017}. Recall that Mellor proved that
the generalized Alexander polynomial dominates the writhe polynomial. 
Specifically, Theorem 1 of \cite{Mellor-2016} asserts that 
\begin{eqnarray*}
W_K(t) &=& -(\De^0_K)'(t^{-1},t),
\end{eqnarray*} 
where $(\De_K^0)'(s,t)={\De_K^0(s,t)}/{(1-s t)}$. (Recall that for virtual knots, $\De_K^0(s,t)$ is divisible by $(1-st)$ 
 by \cite[Proposition 4.1]{Silver-Williams-2003}.)  Corollary \ref{cor_main} therefore implies that the writhe polynomial $W_K(t)$ vanishes on slice virtual knots. 
Since the writhe polynomial is known to be additive under connected sum (see \cite[Proposition 4.9]{Cheng-2014}), it follows
that $W_{K}(t)=W_{K'}(t)$ whenever $K, K'$ are concordant virtual knots.
This argument also shows that the Henrich-Turaev polynomial and the odd-writhe polynomial are concordance invariants for virtual knots, since both are determined by the writhe polynomial.
  

\subsection{Concluding remarks} \label{subsec:conclusion}
Bar-Natan's $\Zh$-construction gives a new interpretation of the generalized Alexander invariant, and it provides a new topological model for the virtual Alexander invariant. 
The association $K \mapsto \Tube(\Zh(K))$ that was used to prove Theorem \ref{thm_main} suggests that techniques from classical link theory can be profitably used to inform the study of virtual knots.
Our main result used only welded invariance of $\Zh(K)$, but there is quite possibly extra information that can be obtained by regarding $\Zh(K)$ as a semi-welded knotted object \cite{Dror}.
It would be interesting to develop a general theory of $(n,m)$ semi-welded knotted objects (links, braids, string links, etc.).

The $\Zh$-map provides a new perspective on the generalized Alexander polynomial, and it may lead to progress on several open problems. One is to develop a better understanding of the behavior of the generalized Alexander polynomial under cut-and-paste operations like connected sum and cabling. Another is the realization problem, which asks which polynomials $\De(s,t) \in \ZZ[s^{\pm 1}, t^{\pm }]$  occur as the generalized Alexander polynomial for some virtual knot (or link). 
A more ambitious problem is to construct a knot homology theory for virtual knots that categorifies the generalized Alexander polynomial. It would be interesting to apply the $\Zh$-construction to these and other problems in virtual knot theory.

Theorem \ref{thm-nil-quot} suggests a potentially new definition of the Milnor $\bar{\mu}$-invariants for virtual links, one for which they are invariant under welded concordance (cf. \cite{Casson-1975}). In future work, the second-named author will combine this approach with the $\Zh$-map to derive new invariants of virtual knot concordance from the lower central series of $\wbar{G}_K$.

\subsection*{Acknowledgements} 
The authors wish to express their gratitude to Dror Bar-Natan, who conceived of the $\Zh$-construction and generously shared it with us. We are also thankful to Heather Dye, Robin Gaudreau, Andrew Nicas, Will Rushworth, Robb Todd and Lindsay White for their valuable input and feedback.

\end{document}